\documentclass[a4paper,reqno,11pt]{amsart}
\usepackage{amssymb, amsfonts}
\usepackage{fancyhdr, mathrsfs, graphicx, epsfig, amsmath, amsthm}
\usepackage{float}
\usepackage{enumitem}
\usepackage{amscd}
\usepackage{tikz, multicol}
\usepackage[colorlinks,linkcolor=blue,anchorcolor=blue,citecolor=red]{hyperref}
\usetikzlibrary{arrows,decorations.pathmorphing}
\usetikzlibrary{chains}

\theoremstyle{plain}
 \newtheorem{theorem}{Theorem}
 \newtheorem{proposition}[theorem]{Proposition}
 \newtheorem{lemma}[theorem]{Lemma}

 \newtheorem{open problem}[theorem]{Open problem}
 \newtheorem{proposed project}[theorem]{Proposed project}
\theoremstyle{definition}
 
 \newtheorem{definition}[theorem]{Definition}
\theoremstyle{remark}
 \newtheorem{remark}[theorem]{Remark}
 
 \numberwithin{equation}{section}

\title{K\"{a}hler cone-manifolds arising from a projective arrangement}
\author{Dali Shen}
\date{\today}
\email{dali\_shen@hotmail.com}

\begin{document}

\setcounter{page}{1}
\pagenumbering{arabic}

\maketitle

\begin{abstract}
Given a hyperplane arrangement of some type in a projective space, the Dunkl system, 
developed by Couwenberg, Heckman and Looijenga, is used to study the geometric structures on its 
complement, and as a consequence it leads to the discovery of new ball quotients when the Schwarz 
conditions are imposed. In this paper, we study how the space, investigated in this system, looks like when there is no Schwarz conditions 
imposed. As a result, we show that the space in question is still of a particular type of structure, namely, the structure 
of a cone-manifold. With this point of view, Thurston's work on moduli of cone metrics on the sphere appears as 
a special case in this set-up.
\end{abstract}

\tableofcontents

\section{Introduction}

The geometric structures on the complement of a given projective arrangement can be 
studied within the framework of the Dunkl system, which was established 
by Couwenberg, Heckman and Looijenga in \cite{Couwenberg-Heckman-Looijenga}. 
A principal consequence of studying the system is the discovery of a few new ball quotients when the Schwarz 
conditions are imposed. That work generalizes the perspective of the work of Barthel-Hirzebruch-H\"{o}fer 
\cite{Barthel-Hirzebruch-Hofer}, which 
investigates the Galois covering of $\mathbb{P}^{2}$ ramified (in a specified manner) over a configuration of lines, 
to a higher dimensional situation so that the Deligne-Mostow theory on Lauricella hypergeometric functions appears as a special case 
in this setting. 

One may wonder, however, what the space, investigated in this system, looks like if the (rather restrictive) Schwarz conditions 
are not imposed. 
This paper is to address this problem and the space in question turns out to be still of some particular type of structures which 
can be understood to some extent. 
Namely, they are actually cone-manifolds. From this point of view, Thurston's work on moduli of cone metrics on 
$\mathbb{P}^{1}$ \cite{Thurston} (another interpretation of the Deligne-Mostow theory) appears as a special case in this setting.

We start out with a quite simple package of data: a finite dimensional complex vector space $V$ with an inner product and a finite 
collection of (linear) hyperplanes in it (called a \emph{hyperplane arrangement}). We write $V^{\circ}$ for the complement in $V$ of the 
union of these hyperplanes. Then we can endow a connection of a specified type, imitating Dunkl's construction for the reflecting 
hyperplane arrangement of a finite Coxeter group, on the tangent bundle of $V^{\circ}$ such that if we regard the connection as a 
meromorphic connection on the tangent bundle of the projective completion of $V$, then it has a logarithmic singularity along each 
member of the arrangement.

The connection is automatically $\mathbb{C}^{\times}$-invariant and torsion free. Moreover if the connection is flat, 
which nevertheless imposes strong conditions on the arrangement, then the connection defines an affine structure on $V^{\circ}$ 
and hence a new projective structure (or affine structure with occasional cases, see Section $\ref{sec:Dunkl-system}$ 
for the precise distinctions) on its projectivization $\mathbb{P}(V^{\circ})$ with an infinitesimally simple degeneration 
along each member of the descended projective arrangement. We call such a system a \emph{Dunkl system}. 
We note that there do exist plenty of such systems, although it is not always the case 
for an arbitrary hyperplane arrangement.

In the meantime, in many cases there also exists a flat admissible Hermitian form on $V^{\circ}$ such that $\mathbb{P}(V^{\circ})$ inherits 
a Hermitian form which is positive definite. When the Hermitian form on $V^{\circ}$ is positive definite, positive semidefinite and 
of hyperbolic signature, it endows $\mathbb{P}(V^{\circ})$ with a K\"{a}hler metric locally isometric to a Fubini-Study metric, 
a flat metric and a hyperbolic metric respectively. We refer to these three situations as \emph{elliptic}, \emph{parabolic} and 
\emph{hyperbolic} case respectively. There may also exist a finite group such that the concerned Dunkl system is 
invariant under the action of this group. We denote this symmetry group by $G$. 
It is proved in \cite{Couwenberg-Heckman-Looijenga} that the space 
$\mathbb{P}(G\backslash V^{\circ})$ can be metrically completed to an orbifold of elliptic, parabolic and hyperbolic type respectively 
if the so-called Schwarz conditions are satisfied.

In this paper we study that in general, without the restriction of the Schwarz conditions, how the space 
$\mathbb{P}(G\backslash V^{\circ})$ looks like. As we already mentioned, 
it turns out that the space in question still acquires the structure of a particular type 
which is, to some extent, understandable. For the precise statements of the main result we refer to Theorems $\ref{thm:elliptic}$, 
$\ref{thm:parabolic}$ and $\ref{thm:hyperbolic}$ for the elliptic, parabolic and hyperbolic case respectively. 
It suffices to say here that the metric completion of $\mathbb{P}(G\backslash V^{\circ})$ acquires the structure of 
a cone-manifold of elliptic, parabolic and hyperbolic type for each case respectively.

The proof relies on a careful analysis on the new projective (or affine) structure (induced by the connection) around 
those (blow-ups of) intersections taken 
by the members of the arrangement, and hence we can see how the metric degenerates from the 
smooth K\"{a}hler metric endowed on the complement of the arrangement. 
It turns out that around those strata in the metric completion of $\mathbb{P}(G\backslash V^{\circ})$, 
obtained by a blow-up followed by a blow-down process, the metric degenerates as of a cone type which can be seen by an inductive manner. 

In addition, we also want to mention that given the modular interpretation for the space associated with some arrangement, 
these cone-manifold structures potentially 
provide a differential-geometric interpretation for the Hassett-Keel-Looijenga program for some moduli spaces.

In fact, on the other hand, in the affine level, if the arrangement complement $V^{\circ}$, whose affine structure is given by 
the flat torsion free connection of the above type, is endowed with a flat K\"{a}hler metric, 
then its metric completion is $V$ and the metric is completed to a polyhedral K\"{a}hler metric on $V$ 
(closely related to the parabolic cone-manifold in our terminology). 
This was done by de Borbon and Panov \cite{DeBorbon-Panov}.
They can make it even more general 
by lifting the self-adjointness condition in the Dunkl system so that angles larger than $2\pi$ (in the parabolic case) 
can be allowed.

The paper is organized as follows. In Section $\ref{sec:cone-manifolds}$ we give a brief introduction to cone-manifolds. 
In Section $\ref{sec:Dunkl-system}$ we set up the Dunkl system. In Section $\ref{sec:cone-manifolds-structres}$ 
we show that the space in question $\mathbb{P}(G\backslash V^{\circ})$ can be metrically completed to a cone-manifold 
for elliptic, parabolic and hyperbolic case respectively, and we also discuss some quantities independent of the type. 
In the last section, i.e., Section $\ref{sec:examples}$, we revisit the Deligne-Mostow theory on Lauricella hypergeometric 
functions (or equivalently, Thurston's work on moduli of cone metrics on $\mathbb{P}^{1}$) with the angle presented in this paper.

\vspace{5mm}
\textbf{Acknowledgements.} I am grateful to Dmitri Panov and Martin de Borbon for the helpful discussions, especially for 
pointing out a gap in the first version of this paper.

\section{Preliminaries on cone-manifolds}\label{sec:cone-manifolds}

In this section we give a quick introduction to cone-manifolds. Here we still follow Thurston's definition on 
$(X,G)-$cone-manifolds \cite{Thurston} since it suffices to describe our cases and what's more, it illustrates 
the uniform local situation (i.e., the pair $(X,G)$) for our spaces. 
A generalized definition on Riemannian cone-manifolds (locally isometric to a $(X,G)-$cone-manifold while the pair $(X,G)$ 
can vary from point to point) can be found in \cite{McMullen}.

Roughly speaking, a cone-manifold is a manifold endowed with a particular kind of singular Riemannian metric. 
In order to see that, let us first recall 
what a spherical complex is, which provides a model for the unit tangent sphere at the singular point on a cone-manifold. 
A \emph{cell complex} is a Hausdorff topological space $M$ together with a partition into disjoint open cells 
$M_{\alpha}$ (of varying dimensions) that satisfies the following two properties:
\begin{enumerate}
	\item For each $n$-dimensional cell, its closure is a homeomorphism onto the $n$-dimensional closed ball in $\mathbb{R}^{n}$, with 
	the cell itself homeomorphic to the interior of the closed ball.
	
	\item Whenever $M_{\alpha}\cap \overline{M_{\beta}}\neq\emptyset$ we have $M_{\alpha}\subset \overline{M_{\beta}}$.
\end{enumerate}

Let $S^{n}$ denote the standard unit sphere in $\mathbb{R}^{n+1}$, a \emph{convex spherical polyhedron} is defined as 
a set $K\subset S^{n}$ obtained by intersecting finitely many hemispheres. 
Then a \emph{spherical complex} is a cell complex glued together by spherical polyhedra in such a way that 
the inclusions $M_{\alpha}\subset\overline{M_{\beta}}$ are isometries. We require that each face of a cell is still a cell. 
We note that every cell in a spherical complex is \emph{totally geodesic}, by which we mean its second fundamental form 
vanishes under every inclusion $M_{\alpha}\subset\overline{M_{\beta}}$.

Given two pointed metric spaces $(X,x_{0})$ and $(Y,y_{0})$, the \emph{pointed Gromov-Hausdorff distance} between these two spaces 
is defined as 
\[
d_{GH}(X,x_{0};Y,y_{0}):=\inf_{(Z,\varphi,\psi)}(d(\varphi(X),\psi(Y))+d(\varphi(x_{0}),\psi(y_{0})))
\]
where $Z$ is a metric space with the metric denoted by $d$, and $\varphi:X\rightarrow Z$, $\psi:Y\rightarrow Z$ are isometric embeddings. 
Let $B_{X}(x,r)$ denote the ball of center $x$ with radius $r$ in the metric space $X$. Then a sequence $\{(X_{n},x_{n})\}$ 
is said to converge to $(X,x)$ in the pointed Gromov-Hausdorff sense if for every $r,\epsilon >0$ there exists an $m\in \mathbb{N}$ 
such that for any $n>m$ the following holds
\[
d_{GH}(B_{X}(x,r);B_{X_{n}}(x_{n},r))<\epsilon.
\]
A \emph{tangent cone} of a metric space $X$ at a point $p\in X$ is then a pointed metric space $(T,o)$ such that 
there exists a sequence of scale factors $\lambda_{n}\rightarrow \infty$ such that 
\[
(\lambda_{n}X,p)\rightarrow (T,o)
\]
in the pointed Gromov-Hausdorff sense as $n\rightarrow \infty$. We note that in general a tangent cone does not always 
exist; and even if it does, it may not be unique. But for our situations, i.e., the underlying space being a manifold 
and endowed with a continuous Riemannian metric, 
the tangent cone does exist (homeomorphic to a vector space) and is unique (e.g., see Proposition 3.15 of \cite{Gromov}).

Now let $X$ be a complete connected $n$-dimensional 
Riemannian manifold and $G$ be a group of isometries of $X$. Then an \emph{$(X,G)-$manifold} is a space $M$ equipped with an atlas of 
charts with homeomorphisms into $X$ such that the transition maps lie in $G$. Let $G_{p}$ denote the stabilizer of $p$ and 
$X_{p}$ the unit tangent sphere of $T_{p}X$. Then $(X_{p},G_{p})$ is a model space of dimension one lower. Following \cite{Thurston}, 
an \emph{$(X,G)-$cone-manifold} can be defined inductively by dimension as follows. 
\begin{definition}
If $\dim X=1$, an $(X,G)-$cone-manifold is nothing but 
an $(X,G)-$manifold. If $\dim X>1$, an $(X,G)-$cone-manifold is a space $M$ such that each point $x\in M$ has a neighborhood 
modelled, using the geometry of $X$, on the cone over a compact and connected $(X_{p},G_{p})-$cone-manifold $S_{x}M$.
We usually call the set of unit tangent vectors at a given point the \emph{unit tangent cone}, 
so that we note that in this sense $S_{x}M$ is just the unit tangent cone to $M$ at $x$.
\end{definition}

According to this, a \emph{spherical cone-manifold} is an $(S^{n},\mathrm{O}(n+1))-$cone-manifold. 
Let $M$ and $N$ be two compact spherical 
cone-manifolds. Assuming that each of them is connected unless it is isomorphic to $S^{0}$. We can define the \emph{join} 
of $M$ and $N$, denoted by $M*N$, as a spherical cone-manifold obtained from the disjoint union of $M$ and $N$ by adding 
a spherical arc $[x,y]$ of length $\pi/2$ between every pair of points $x\in M$ and $y\in N$. For instance, $S^{m}*S^{n}
\cong S^{m+n+1}$. The unit tangent cones of the join are given as follows
\begin{equation*}
S_{z}(M*N)\cong
\left\{
\begin{aligned}
&(S_{z}M)*N  && \text{if $z\in M$,} \\
&M*(S_{z}N)  && \text{if $z\in N$, and} \\
&S^{0}*(S_{x}M)*(S_{y}N)  && \text{if $z$ lies in the open arc $(x,y)$}.
\end{aligned}
\right.
\end{equation*}

We say that $M$ is \emph{prime} if it can not be expressed by a join $M\cong S^{0}\ast N$. Then we have the following 
factorization theorem.
\begin{theorem}[{\cite[Theorem 5.1]{McMullen}}]
A compact spherical cone-manifold $M$ can be canonically expressed as the join $M\cong S^{k}* N$ of a sphere 
and a prime cone-manifold $N$.
\end{theorem}

Given an $(X,G)-$cone-manifold $M$ of dimension $n$. 
Locally, there is an exponential map from a small ball $B^{n}$ in $T_{p}X$ to the neighborhood $U$ of $x$ in $M$
\[
\exp: B^{n}\rightarrow M.
\]
By pulling back the metric on $M$ the small ball $B^{n}$ is endowed with a metric, denoted by $g$. Then the metric cone 
$(C_{g}(J),0)$ over $J$, where $J\cong S_{x}M$, 
provides an isometric local model for $(U,x)$. Indeed, it can be formed by 
assembling those polyhedral cones over the cells of $S_{x}M$.

Its $k$-dimensional strata are the connected components of the 
locus characterized by 
\[
M[k]=\{x\in M\mid S_{x}M\cong S^{k-1}*N, \text{with $N$ prime}\}.
\]
We usually call the set of unit normal vectors to a stratum at a given point $x$ the \emph{normal cone}, which will be denoted by 
$N_{x}M$ in what follows. In this sense we note that the prime factor $N$ above is just the normal cone to $M[k]$ at $x$, 
and thus we can rewrite the above decomposition as 
\[
S_{x}M\cong S^{k-1}*N_{x}M
\]
for all $x\in M[k]$.

We have the local model $(C_{g}(S_{x}M),0)$ for a neighborhood $U$ of $x\in M[k]$, it is not hard to see that 
\[
M[k]\cap U\cong C_{g}(S^{k-1}).
\]
This also shows that the geodesic ray from $x$ to any nearby point $y\in M[k]$ entirely lies in $M[k]$ so that 
$M[k]$ is totally geodesic in $M$. 
It is clear that the loci $M[k]$ ($0\leq k\leq n$) gives rise to a stratification for $M$ for which the top-dimensional 
stratum $M[n]$ is open and dense in $M$. Note that $M[n-1]$ is empty, and thus $M$ is the metric completion of $M[n]$.

Now let us write the cone metric in terms of coordinates, at least locally. We take the flat metric with a cone singularity along 
a real codimension 2 stratum for example.  
Write $\mathbb{R}^{n}=\mathbb{R}^{2}\times\mathbb{R}^{n-2}$ 
and let $D$ denote the codimension 2 submanifold $\{0\}\times\mathbb{R}^{n-2}$. Take polar coordinates $r,\theta$ on $\mathbb{R}^{2}$ 
and standard coordinates $u_{i}$ on $\mathbb{R}^{n-2}$. Fix $\beta\in (0,1)$ and the following singular metric 
\[
g=dr^{2}+\beta^{2}r^{2}d\theta^{2}+\sum du_{i}^{2}
\]
is called the \emph{standard cone metric} with cone angle $2\pi\beta$ along $D$.

If we take a complex point of view, let $\mathbb{C}_{\beta}$ denote the complex plane $\mathbb{C}$ endowed with the metric 
$\beta^{2}|\zeta|^{2(\beta-1)}|d\zeta|^{2}$ and let $n=2m$, then there is an isometry from the above singular metric 
on $\mathbb{R}^{n}$ to the Riemannian product $\mathbb{C}_{\beta}\times\mathbb{C}^{m-1}$, realized by taking the coordinate 
$\zeta=r^{\frac{1}{\beta}}e^{\sqrt{-1}\theta}$ on $\mathbb{C}_{\beta}$ and taking standard complex coordinates on $\mathbb{C}^{m-1}$. 
This indeed describes the flat K\"{a}hler metric with a cone singularity along a smooth hypersurface. 
It also provides a local model for the tangent cone at a generic point of a divisor along which the concerned K\"{a}hler metric 
is of a conical type. For foundations of K\"{a}hler metrics with cone singularities along a divisor, interested readers can 
consult \cite{Donaldson}.

\section{Set-up of the Dunkl system}\label{sec:Dunkl-system}

In this section we review the Dunkl system developed by Couwenberg-Heckman-Looijenga \cite{Couwenberg-Heckman-Looijenga}, 
which studies the geometric structures on projective arrangement complements. This set-up generalizes the work of 
Deligne-Mostow on Lauricella hypergeometric differential equations, as well as the work of Barthel-Hirzebruch-H\"{o}fer on 
line arrangements in a projective plane, and thus leads to discovering a bunch of complete orbifolds of various geometric types 
(in particular, complex hyperbolic orbifolds).

Let us start with a finite dimensional complex vector space $V$ endowed with a positive definite inner product, a finite collection
$\mathscr{H}$ of linear hyperplanes in $V$. Such $(V,\mathscr{H})$ is called a $\emph{linear hyperplane arrangement}$.
We define the \emph{arrangement complement} as $V^{\circ}:=V-\cup_{H\in\mathscr{H}}H$, i.e., the complement of the
union of the members of $\mathscr{H}$ in $V$. We will always use the superscript $^{\circ}$ to denote the complement of
an arrangement in an analogous situation as long as the arrangement is understood.

We denote by $\mathscr{L}(\mathscr{H})$ the collection of hyperplane
intersections taken from subsets of $\mathscr{H}$; here we understand $V$ is also included in $\mathscr{L}(\mathscr{H})$
as the intersection over the empty subset of $\mathscr{H}$. And for $L\in\mathscr{L}(\mathscr{H})$,
we denote by $\mathscr{H}_{L}$ the collection of $H\in\mathscr{H}$ which contains $L$. On the other hand, each
$H\in\mathscr{H}-\mathscr{H}_{L}$ meets $L$ in a hyperplane of $L$. We denote the collection of these hyperplanes
of $L$ by $\mathscr{H}^{L}$. But we notice that the natural map $\mathscr{H}-\mathscr{H}_{L}\rightarrow\mathscr{H}^{L}$
needs not be injective, so that $\mathscr{H}-\mathscr{H}_{L}$ and $\mathscr{H}^{L}$ cannot be identified in general.
We say that $\mathscr{H}$ is \emph{irreducible} if there does not exist a nontrivial decomposition of $\mathscr{H}$ such that
$\mathscr{H}=\mathscr{H}_{L}\sqcup\mathscr{H}_{L'}$ with $L,L'\in\mathscr{L}(\mathscr{H})$ and $L+ L'=V$. A member
$L\in\mathscr{L}(\mathscr{H})$ is also called irreducible if $\mathscr{H}_{L}$ is. In other words, 
there exists a nontrivial decomposition of $\mathscr{H}_{L}$ if a member $L\in\mathscr{L}(\mathscr{H})$ is not irreducible 
(we call it \emph{reducible}). 
We denote the subset of
irreducible members of $\mathscr{L}(\mathscr{H})$ by $\mathscr{L}_{\mathrm{irr}}(\mathscr{H})$.

Suppose we are given a map $\kappa$ which assigns to every $H\in\mathscr{H}$ a positive real number $\kappa_{H}$, then we can
define a connection $\nabla^{\kappa}$ on the tangent bundle of $V^{\circ}$. For $H\in\mathscr{H}$, let $\pi_{H}\in\mathrm{End}(V)$
denote the orthogonal projection onto $H^{\perp}$, then $\rho_{H}:=\kappa_{H}\pi_{H}$ is self-adjoint with respect to the
inner product. So its kernel is $H$ with trace $\kappa_{H}$. We also let $\omega_{H}$ denote the unique meromorphic
differential on $V$ with divisor $-H$ and residue $1$ along $H$. So $\omega_{H}:=\phi_{H}^{-1} d\phi_{H}$, where
$\phi_{H}$ is a defining linear equation for $H$. Put the $\emph{connection form}$ \index{Connection form}
\[ \Omega^{\kappa}:=\sum_{H\in\mathscr{H}}\omega_{H}\otimes\rho_{H}   \]
and view it as an $\mathrm{End}(V)$-valued holomorphic differential on $V^{\circ}$. Then the desired connection
is defined as follows
\[ \nabla^{\kappa}:=\nabla^{0}-\Omega^{\kappa}  \]
where $\nabla^{0}$ denotes the standard (translation invariant) connection on $V$ restricted on $V^{\circ}$.

Now our $\nabla^{\kappa}$ is easily verified to be torsion free, and there are some criteria for the flatness of
this connection.

\begin{proposition}[Flatness criteria]
The connection $\nabla^{\kappa}$ is $\mathbb{C}^{\times}$-invariant and torsion free. Moreover, the following properties are
	equivalent:
	\begin{enumerate}[label=(\roman*)]
		\item $\nabla^{\kappa}$ is flat,
		
		\item $\Omega^{\kappa}\wedge\Omega^{\kappa}=0$,
		
		\item for every pair $L,L'\in\mathscr{L}(\mathscr{H})$ with $L\subset L'$, the endomorphisms $\sum_{H\in\mathscr{H}_{L}}\rho_{H}$
		and $\sum_{H\in\mathscr{H}_{L'}}\rho_{H}$ commute,
		
		\item \emph{(Codimension 2 conditions)} for every $L\in\mathscr{L}(\mathscr{H})$ of codimension $2$, the sum $\sum_{H\in\mathscr{H}_{L}}\rho_{H}$ commutes with
		each of its terms.
	\end{enumerate}
\end{proposition}

If the connection $\nabla^{\kappa}$ is flat, we say that the triple $(V,\mathscr{H},\kappa)$ is a \emph{Dunkl system}.
\index{Dunkl system}
From \cite{Couwenberg-Heckman-Looijenga} we know that besides the Lauricella case (which we will revisit in Section 
$\ref{sec:examples}$), there also exist plenty of other
Dunkl systems of interest, such as the arrangement types associated with complex reflection groups: 
let $G$ be a finite complex reflection group acting irreducibly and unitarily on a complex inner product space $V$, 
and $\mathscr{H}$ the collection of those reflection hyperplanes induced by $G$, assuming that those numbers 
$\kappa_{H}$ associated to every $H\in\mathscr{H}$ are constant on the $G$-orbits, then $(V,\mathscr{H},\kappa)$ is a Dunkl system.

It is well-known that a flat torsion free connection on a tangent bundle defines an affine structure.
We denote the holonomy group of the Dunkl system by $\Gamma$. So both
$V^{\circ}$ and its $\Gamma$-covering
$\widehat{V^{\circ}}$ come with one, denoted by $\mathrm{Aff}_{V^{\circ}}$ resp. $\mathrm{Aff}_{\widehat{V^{\circ}}}$,
that is a subsheaf consisting of affine-linear
functions of the structure sheaf. The sheaf of affine-linear functions
is in fact the sheaf of holomorphic functions
whose differential is flat for the connection, which corresponds to a system of second order differential equations.
Conversely, an affine structure is always given by a flat torsion free connection. Then we see that the space
of affine-linear functions on $\widehat{V^{\circ}}$ is given by $\mathrm{Aff}(\widehat{V^{\circ}}):=H^{0}
(\widehat{V^{\circ}},\mathrm{Aff}_{\widehat{V^{\circ}}})$. We also denote by $A$ the set of
linear forms $\mathrm{Aff}(\widehat{V^{\circ}})\rightarrow\mathbb{C}$ which are the identity on $\mathbb{C}$.
In fact, this is an affine $\Gamma$-invariant hyperplane in $\mathrm{Aff}(\widehat{V^{\circ}})^{*}$. Then the evaluation map
$ev:\widehat{V^{\circ}}\rightarrow A$ which assigns to $\hat{z}$ the linear form $ev_{\hat{z}}:\hat{f}
\in\mathrm{Aff}(\widehat{V^{\circ}})\mapsto \hat{f}(\hat{z})\in\mathbb{C}$ is called the
\emph{developing map} of the affine structure. It is $\Gamma$-equivariant and a local affine isomorphism. Therefore, 
the developing map determines a natural affine atlas on $V^{\circ}$ whose charts take values in $A$ and whose transition maps 
lie in $\Gamma$. We thus call the descended evaluation map $ev:V^{\circ}\rightarrow A$ the \emph{developing map} 
as well, of the affine structure on $V^{\circ}$, which is a multivalued local isomorphism up to $\Gamma$. 

We already know that the connection is invariant under the scalar multiplication by $e^{t}\in \mathbb{C}^{\times}$. 
One can verify that for $t$ close to $0$ enough, those affine-linear functions behave like a homothety action by a 
scalar $e^{(1-\kappa_{0})t}$ when $\kappa_{0}\neq 1$, or like a translation when $\kappa_{0}=1$. 
(See below for the definition of $\kappa_{0}$.) This shows that when 
$\kappa_{0}\neq 1$, the affine structure on $V^{\circ}$ is in fact a linear structure and thus endows on $\mathbb{P}(V^{\circ})$ 
with a corresponding projective structure. This is realized by the projectivization of the developing map 
and we thus call it the \emph{projective developing map} of the projective structure, and so for its projectivized 
$\Gamma$-covering space $\mathbb{P}(\widehat{V^{\circ}})$. 
While when $\kappa_{0}=1$, the projection $V^{\circ}\rightarrow \mathbb{P}(V^{\circ})$ 
is affine linear so that $\mathbb{P}(V^{\circ})$ inherits an affine structure from $V^{\circ}$.

Furthermore, if we want to extend this developing map, we need to understand how the system behaves 
near a given subvariety of its singular locus. For that it is natural to blow up that subvariety so that 
we are in the codimension one case. We thus encounter the simplest degenerating affine structure described as follows. 

\begin{definition}\label{def:degeneration}
Let $E$ be a smooth connected hypersurface in a complex manifold $M$ and an affine structure is endowed on $M-E$. 
We say that the affine structure on $M-E$ has an \emph{infinitesimally simple degeneration along $E$ of 
logarithmic exponent $\lambda\in\mathbb{C}$} if 
\begin{enumerate}[label=(\roman*)]
	
\item $\nabla$ extends to $\Omega_{M}(\log E)$ with a logarithmic pole along $E$,

\item the residue of this extension along $E$ preserves the subsheaf $\Omega_{E}\subset \Omega_{M}(\log E)\otimes\mathcal{O}_{E}$ 
and its eigenvalue on the quotient sheaf $\mathcal{O}_{E}$ is $\lambda$ and 

\item the residue endomorphism restricted to $\Omega_{E}$ is semisimple and all of its eigenvalues are $\lambda$ or $0$.
\end{enumerate}
When in addition
\begin{enumerate}[label=(\roman*),resume]
	\item the connection has semisimple monodromy on the tangent bundle,
\end{enumerate}
then the adjective \emph{infinitesimally} can be dropped and we say that the affine structure on $M-E$ has a 
\emph{simple degeneration along $E$}.
\end{definition}

In our situation, for those intersection members we have some nice hereditary properties of the Dunkl connection. 
For $L\in\mathscr{L}(\mathscr{H})$ irreducible, since $\sum_{H\in\mathscr{H}_{L}}\kappa_{H}\pi_{H}$ commutes with each of its terms, 
we must have that $\sum_{H\in\mathscr{H}_{L}}\kappa_{H}\pi_{H}=\kappa_{L}\pi_{L}$ for some number $\kappa_{L}$,
where $\pi_{L}$ is the orthogonal projection onto $L^{\perp}$. So by a trace computation, we have
\[
\kappa_{L}=\mathrm{codim}(L)^{-1}\sum_{H\in\mathscr{H}_{L}}\kappa_{H}.
\]
From now on we assume that the whole system is irreducible and 
that the common intersection of all members of $\mathscr{H}$ 
is reduced to the origin, hence for the origin, we have
\[
\kappa_{0}=\dim(V)^{-1}\sum_{H\in\mathscr{H}}\kappa_{H}.
\]
In fact, these numbers $\kappa_{L}$ have
an inclusion-reverse property, i.e., $\kappa_{L'}<\kappa_{L}$ if $L\subsetneq L'$. If we denote by $E$ the
exceptional divisor of the blow-up of $L$ in V, then the affine structure on $V^{\circ}$ is of infinitesimally simple type 
along $E^{\circ}$ with logarithmic exponent $\kappa_{L}-1$ (see Lemma 2.21 of \cite{Couwenberg-Heckman-Looijenga}).

Then for every
$L\in\mathscr{L}(\mathscr{H})$, put
\[ \Omega_{L}:=\sum_{H\in\mathscr{H}_{L}}\omega_{H}\otimes\rho_{H}.  \]
This defines a Dunkl connection $\nabla_{L}$ on $(V/L)^{\circ}$.

Notice that for each $I\in\mathscr{H}^{L}$,
there exists a unique irreducible intersection $I(L)\in\mathscr{L}(\mathscr{H})$ such that
$L\cap I(L)=I$. Then we can also define a Dunkl connection $\nabla^{L}$ on $L^{\circ}$ as follows
\[ \Omega^{L}:=\sum_{I\in\mathscr{H}^{L}}\omega_{I}^{L}\otimes \kappa_{I(L)}\pi_{I}^{L}  \]
where $\pi_{I}^{L}$ denotes the restriction of $\pi_{I}$ to $L$.

We thus call the Dunkl connection $\nabla_{L}$ resp. $\nabla^{L}$ defined on $(V/L)^{\circ}$ resp. $L^{\circ}$ the
$L$-$\emph{transversal}$ resp. $L$-$\emph{longitudinal}$ Dunkl connection.

For each given $\kappa$, we can deform the connection and the Hermitian form at the same time so that we can induce
different geometric structures on $\mathbb{P}(V^{\circ})$. First we can deform the standard connection $\nabla^{0}$ in
a real analytic way to a one-parameter family of flat torsion free connections, denoted by $\nabla^{s}$, $s\geq 0$, such that 
it would reach $\nabla^{\kappa}$ for some $s>0$.
On the other hand, the inner product gives rise to a translation invariant metric on $V$. Its restriction $h^{0}$ to
$V^{\circ}$ is flat for $\nabla^{0}$. Then we can also deform $h^{0}$ to a family of Hermitian forms such that each
$h^{s}$ is flat for $\nabla^{s}$. Since the scalar multiplication in $V$ acts locally like homothety, we have that
$\mathbb{P}(V^{\circ})$ inherits a Hermitian metric $g^{s}$ from the one on $V^{\circ}$. We only allow $s$ vary in an interval for which
$g^{s}$ stays positive definite. But meanwhile this still makes it possible for $h^{s}$ to become degenerate or of hyperbolic
type, as long as for every $p\in V^{\circ}$ the restriction of $h^{s}$ to a hyperplane which is supplementary and
perpendicular to $T_{p}(\mathbb{C}p)$ is positive definite. We call the last property the \emph{admissible} 
condition for a Hermitian form. 

We conclude this discussion in the following theorem.

\begin{theorem}[{\cite[Theorem 3.1]{Couwenberg-Heckman-Looijenga}}]\label{thm:hermitian-forms}
Let $\dim V\geq 2$.	Suppose we are given a Dunkl system $(V,\mathscr{H},\kappa)$ such that $\kappa_{0}=1$.
Assume that	there is for every $s\geq 0$ a nonzero Hermitian form $h^{s}$ on $V^{\circ}$ which is flat for $\nabla^{s\kappa}$.
Assume it is real-analytic in $s$ and is equal to the given positive definite form for $s=0$. Then we have:
	\begin{enumerate}[label=(\roman*)]
		\item for $s<1$, $h^{s}>0$,
		
		\item $h^{1}\geq 0$ and its kernel is spanned by the Euler field,
		
		\item there exists a $m_{\mathrm{hyp}}\in(1,\infty]$ such that for $s\in(1,m_{\mathrm{hyp}})$, $h^{s}$ is of hyperbolic type, 
		and $h^{s}$ is admissible as long as $s\kappa_{H}\leq 1$ for all $H\in \mathscr{H}$ (although it is likely that $h^{s}$ is 
		admissible for all $s\in(1,m_{\mathrm{hyp}})$).
	\end{enumerate}
	We call $m_{\mathrm{hyp}}$ the $\emph{hyperbolic exponent}$ of the family.
\end{theorem}

It turns out that in many cases of interest (including the examples mentioned above), such Hermitian forms $h$ do exist. 
Then the Hermitian form endows $\mathbb{P}(V^{\circ})$ with a natural \emph{K\"{a}hler} metric which is $\Gamma$-invariant, 
and locally	isometric to 
\begin{itemize}
	\item[\textbf{ell:}] a Fubini-Study metric, if $h$ is positive definite,
	
	\item[\textbf{par:}] a flat metric, if $h$ is semidefinite with kernel $T_{p}(\mathbb{C}p)$,
	
	\item[\textbf{hyp:}] a complex hyperbolic metric, if $h$ is nondegenerate of hyperbolic signature, and negative on  $T_{p}(\mathbb{C}p)$.
\end{itemize}
We refer to these three situations as the \emph{elliptic}, \emph{parabolic} and \emph{hyperbolic} case respectively.
They are all of constant holomorphic sectional curvatures: positive, zero, negative respectively, 
and hence their curvature models in terms of local coordinates (suppose 
$\dim V=n+1$) are given by 
\begin{align*}
\sqrt{-1}\partial\bar{\partial}\log (1+\sum_{i=1}^{n}|z_{i}|^{2}),\quad \sqrt{-1}\partial\bar{\partial} (\sum_{i=1}^{n}|z_{i}|^{2}), \quad
-\sqrt{-1}\partial\bar{\partial}\log (1-\sum_{i=1}^{n}|z_{i}|^{2}).
\end{align*}

For us it is also important to write the local models around the singular locus if along which these metrics have cone singularities. 
It is clear that for each case the tangent cone at a point of the singular locus can be described by a flat model. 
We now take the simple normal crossing locus for example (also discussed in \cite[Section 4]{DeBorbon-Spotti}), suppose it is written as 
$\{z_{1}\cdots z_{k}=0\}$ in local coordinates, the tangent cone must be isomorphic to 
$\mathbb{C}_{\beta_{1}}\times\cdots \times\mathbb{C}_{\beta_{k}}\times \mathbb{C}^{n-k}$ (a direct consequence from the discussion 
on the flat model for the divisor case in the preceding section), then the constant holomorphic sectional curvature models 
are given by 
\begin{align}
&\sqrt{-1}\partial\bar{\partial}\log (1+\sum_{i=1}^{k}|z_{i}|^{2\beta_{i}}+\sum_{i=k+1}^{n}|z_{i}|^{2}), \label{eqn:positive} \\ 
&\sqrt{-1}\partial\bar{\partial} (\sum_{i=1}^{k}|z_{i}|^{2\beta_{i}}+\sum_{i=k+1}^{n}|z_{i}|^{2}), \label{eqn:zero} \\
&-\sqrt{-1}\partial\bar{\partial}\log \big(1-(\sum_{i=1}^{k}|z_{i}|^{2\beta_{i}}+\sum_{i=k+1}^{n}|z_{i}|^{2})\big), \label{eqn:negative}
\end{align}
for the positive, zero and negative case respectively.

Depending on the symmetry invoked by the function $\kappa$ and the configuration of hyperplanes, 
there also exists a group $G$ which is a finite subgroup of the unitary group $U(V)$ 
such that under the action of this group the Dunkl system is invariant. We call such a group the 
\emph{Schwarz symmetry group}.

If the so-called \emph{Schwarz condition} (see Definition 4.2 of \cite{Couwenberg-Heckman-Looijenga} for the full description) 
is imposed on the system (this is the generalization of \emph{$\Sigma$INT condition} 
in the theory of Deligne-Mostow on the Lauricella functions, or of \emph{orbifold condition} 
in Thurston's work on moduli of cone metrics on $\mathbb{P}^{1}$), Couwenberg, Heckman and Looijenga \cite{Couwenberg-Heckman-Looijenga} 
showed that the space in question $\mathbb{P}(G\backslash V^{\circ})$ can be metrically completed to an 
elliptic, parabolic and hyperbolic orbifold respectively.

\section{Cone-manifolds arising from the Dunkl system}\label{sec:cone-manifolds-structres}

In this section we treat the parabolic, elliptic and hyperbolic cases respectively when there is no Schwarz condition imposed. 
It turns out that the space in question $\mathbb{P}(G\backslash V^{\circ})$ can still be metrically completed to a kind of spaces, 
namely, the cone-manifolds, introduced in Section $\ref{sec:cone-manifolds}$.

From now on the complex dimension of $V$ is always set to be $n+1$, hence $\dim_{\mathbb{C}}\mathbb{P}(V)=n$.

\subsection{Parabolic case}

We first treat the parabolic case. This case is relatively easy to deal with while it is also of its own interest, 
since it provides a model for the tangent cone of all three geometric types treated in this paper (including for itself).

The main result for the parabolic case is as follows.

\begin{theorem}\label{thm:parabolic}
	Let be given a Dunkl system with $\kappa_{H}\in (0,1)$ for every $H\in \mathscr{H}$. 
	Suppose that $\kappa_{0}=1$ and that it admits a flat positive semidefinite Hermitian form on the tangent bundle of $V^{\circ}$. 
	Then the projective (``restricted" in this case, actually) developing map realizes the 
	space $\mathbb{P}(G\backslash V^{\circ})$ as a divisor complement of 
	its metric completion, via which $\mathbb{P}(G\backslash V^{\circ})$ acquires the structure of a parabolic cone-manifold 
	(i.e., $(\mathbb{C}^{n},\mathrm{U}(n))-$cone-manifold). In this case, the metric completion of $\mathbb{P}(G\backslash V^{\circ})$ 
	is $\mathbb{P}(G\backslash V)$.
\end{theorem}

Before we proceed to the proof, we need to understand how the projective developing map extends near 
those members of $\mathscr{L}_{\mathrm{irr}}(\mathscr{H})$. 
For that we introduce a (so-called \emph{big}) resolution of $V$ (hence also of $\mathbb{P}(V)$). 

Our resolved space $V^{\sharp}$ is obtained by blowing up the members of $\mathscr{L}_{\mathrm{irr}}(\mathscr{H})$ 
in their natural partial order: we first blow up the dimension $0$ members, then the strict transform of the dimension $1$ members, 
and so on, until the strict transform of codimension $2$ members. We will identify $V^{\circ}$ with its preimage in $V^{\sharp}$. 
Notice that the action of the group $G$ can be naturally extended to $V^{\sharp}$.

Now every $L\in\mathscr{L}_{\mathrm{irr}}(\mathscr{H})$ determines an exceptional divisor $E(L)$ and all of these 
exceptional divisors form a normal crossing divisor in $V^{\sharp}$. Since $\mathscr{H}^{L}$ (resp. $\mathscr{H}_{L}$) 
defines an arrangement in $L$ (resp. $V/L$, and hence in $\mathbb{P}(V/L)$), 
the divisor $E(L)$ can be identified with $L^{\sharp}\times \mathbb{P}((V/L)^{\sharp})$ where $L^{\sharp}$ denotes 
the strict transform of $L$ until before blowing up itself (so for $\mathbb{P}((V/L)^{\sharp})$).
Moreover, the divisor $E(L)$ contains an open dense stratum $L^{\circ}\times\mathbb{P}(V/L)^{\circ}$, denoted by $E(L)^{\circ}$. 
We already know that the affine structure on $V^{\circ}$ degenerates 
infinitesimally simply along $E(L)^{\circ}$ (in the sense of Definition $\ref{def:degeneration}$) with 
logarithmic exponent $\kappa_{L}-1$, so that the behavior of the developing map near $E(L)^{\circ}$ can be understood 
in an explicit manner: assume $\kappa_{L}-1\notin\mathbb{Z}$ for the moment, 
write a point $z=(z_{0},z_{1})\in E(L)^{\circ}$, let $V_{z}^{\sharp}$ denote the germ of $V^{\sharp}$ at $z$ and so for $L_{z_{0}}$, 
etc., then there exist a submersion 
$F_{0}:V_{z}^{\sharp}\rightarrow L_{z_{0}}$, 
a submersion $F_{1}:V_{z}^{\sharp}\rightarrow T_{1}$ (with $T_{1}$ a vector space) and a defining equation $t$ for $E(L)_{z}^{\circ}$ 
such that $(F_{0},t,F_{1})$ is a chart for $V_{z}^{\sharp}$, and the developing map is affine equivalent to the map 
\[
(F_{0},t^{1-\kappa_{L}},t^{1-\kappa_{L}}F_{1}):V_{z}^{\sharp}\rightarrow L_{z_{0}}\times \mathbb{C}\times T_{1}.
\]

Furthermore, the divisors $E(L)$ also determine a stratification of $V^{\sharp}$. An arbitrary stratum can be described as follows: 
the intersection $\cap_{i}E(L_{i})$ is nonempty if and only if these members make up a flag: 
$L_{\bullet}:=L_{0}\subset L_{1}\subset \cdots \subset L_{k}\subset L_{k+1}=V$. Their intersection $E(L_{\bullet})$ 
can be identified with the product 
\[
E(L_{\bullet})=L_{0}^{\sharp}\times\mathbb{P}((L_{1}/L_{0})^{\sharp})\times\cdots\times\mathbb{P}((V/L_{k})^{\sharp}),
\]
and it contains an open dense stratum $E(L_{\bullet})^{\circ}$ decomposed as follows 
\[
E(L_{\bullet})^{\circ}=L_{0}^{\circ}\times\mathbb{P}((L_{1}/L_{0})^{\circ})\times\cdots\times\mathbb{P}((V/L_{k})^{\circ}).
\]
For a point $z=(z_{0},z_{1},\dots,z_{k+1})\in E(L_{\bullet})^{\circ}$, there exist a submersion 
$F_{0}:V_{z}^{\sharp}\rightarrow (L_{0})_{z_{0}}$, 
a submersion $F_{i}:V_{z}^{\sharp}\rightarrow T_{i}$ (with $T_{i}$ a vector space, identified with $\mathbb{P}(L_{i}/L_{i-1})_{z_{i}}$) 
and a defining equation $t_{i-1}$ of $E(L_{i-1})$ for each $i=1,\dots,k+1$ such that $\big(F_{0},(t_{i-1},F_{i})_{i=1}^{k+1}\big)$ is 
a chart for $V_{z}^{\sharp}$, and the developing map is affine equivalent to the map
\[
\Big(F_{0},\big(t_{0}^{1-\kappa_{0}}\cdots t_{i-1}^{1-\kappa_{i-1}}(1,F_{i})\big)_{i=1}^{k+1}\Big):V_{z}^{\sharp}\rightarrow (L_{0})_{z_{0}}
\times \prod_{i=1}^{k+1}(\mathbb{C}\times T_{i}),
\]
where we write $\kappa_{i}$ for $\kappa_{L_{i}}$.

Notice that the developing map would not extend unless we projectivize (when $1-\kappa_{L}<0$ for some $L$). 
So we will focus on the central exceptional divisor 
$E(\{0\})$, which records all the birational operations in the projective level descended from the birational morphism 
$V^{\sharp}\rightarrow V$. We thus 
denote this space by $\mathbb{P}(V^{\sharp})$. This resolved space $\mathbb{P}(V^{\sharp})$ for $\mathbb{P}(V)$ is a projective manifold. 
Notice that each $E(L)$ for $L\neq \{0\}$ meets $\mathbb{P}(V^{\sharp})$ in a smooth hypersurface of $\mathbb{P}(V^{\sharp})$, 
denoted by $D(L)$, and these hypersurfaces form a normal crossing divisor in 
$\mathbb{P}(V^{\sharp})$. It is clear that $\mathbb{P}(V^{\sharp})$ 
contains $\mathbb{P}(V^{\circ})$ as an open dense stratum. And a stratum of $\mathbb{P}(V^{\sharp})$ is given by a flag 
$L_{\bullet}$ as described above provided that $L_{0}=\{0\}$, (with this point of view, the exceptional 
divisor $D(L)^{\circ}$ is just the stratum $E(L_{\bullet})^{\circ}$ with the flag $L_{\bullet}:L_{0}=\{0\}\subset
L_{1}:=L\subset V$.) then the developing map at a point 
$z=(z_{1},\dots,z_{k+1})\in E(L_{\bullet})^{\circ}$ is linearly equivalent to the map 
\[
\big(t_{0}^{1-\kappa_{0}}\cdots t_{i-1}^{1-\kappa_{i-1}}(1,F_{i})\big)_{i=1}^{k+1}:V_{z}^{\sharp}\rightarrow 
\prod_{i=1}^{k+1}(\mathbb{C}\times T_{i}),
\]
since now $L_{0}$ is just a point.

For our purpose we need to consider the projective developing map restricted on $\mathbb{P}(V^{\sharp})$ (which can also 
be realized by letting $t_{0}=0$), then the preceding map is projectively equivalent to the map 
\begin{align}\label{eqn:Pev-mapping}
\Big[(1,F_{1}),\big(t_{1}^{1-\kappa_{1}}\cdots t_{i-1}^{1-\kappa_{i-1}}(1,F_{i})\big)_{i=2}^{k+1}\Big].
\end{align}

Under the assumption of Theorem $\ref{thm:parabolic}$: there admits a flat positive semidefinite Hermitian form on the tangent bundle 
of $V^{\circ}$, so that the Hermitian form endows $\mathbb{P}(V^{\circ})$ with a flat K\"{a}hler metric. 
Now let us find out what the metric completion of $\mathbb{P}(G\backslash V^{\circ})$, induced from $\mathbb{P}(V^{\circ})$, 
is for this case, which is required to be compatible with its affine structure induced by the Dunkl connection. 

For that we first need to introduce the slightly modified \emph{Stein factorization} for a complex (normal) analytic space, 
which will be later on applied to our $\mathbb{P}(G\backslash V^{\sharp})$ for all the three cases. Let us start from an 
easy version in the topological category.

\begin{definition}[Topological Stein factorization]
Let $f:X\rightarrow Y$ be a continuous map between topological spaces. The \emph{topological Stein factorization} of 
the map $f$ is the factorization through the quotient $X\rightarrow X_{f}$ of $X$ defined by the partition of $X$ 
into connected components of fibers of $f$. The second map is denoted $f_{\mathfrak{St}}: X_{f}\rightarrow Y$ and 
we usually write $X_{f}$ as $X_{\mathfrak{St}}$ if $f$ is understood. 
We call $X_{\mathfrak{St}}$ the \emph{Stein factor} of the factorization.
\end{definition}

So the first map $X\rightarrow X_{f}$ has connected fibers and the second map $f_{\mathfrak{St}}: X_{f}\rightarrow Y$ 
has discrete fibers in case the fibers of $f$ are locally connected. We have a useful criterion as follows for a 
complex-analytic counterpart.

\begin{lemma}[Stein Factorization Theorem, {\cite[Lemma 5.13]{Couwenberg-Heckman-Looijenga}}]\label{lem:Stein-factorization}
Let $f:X\rightarrow Y$ be a morphism of connected normal analytic spaces. Suppose that the connected components of the 
fibers of $f$ are compact. Then in the complex-analytic category the morphism $f$ admits a unique factorization 
\[
f:X\longrightarrow X_{\mathfrak{St}}\stackrel{f_{\mathfrak{St}}}\longrightarrow Y,
\]
through a connected normal analytic space $X_{\mathfrak{St}}$, 
such that $X\rightarrow X_{\mathfrak{St}}$ is a proper morphism with connected 
fibers to $X_{\mathfrak{St}}$ and $f_{\mathfrak{St}}$ is a morphism with discrete fibers. 
\end{lemma}

The following lemma also applies to all the three cases (for the situation without boundary components).

\begin{lemma}\label{lem:compact-fiber}
The connected components of the fibers of the projective developing map from $\mathbb{P}(V^{\sharp})$ to the corresponding target space 
are compact (and hence so for the induced map from $\mathbb{P}(G\backslash V^{\sharp})$ to the target space).
\end{lemma}

\begin{proof}
Over $\mathbb{P}(V^{\circ})$ the projective developing map is locally finite and hence the claim on these points is clear. 
So let us examine the situation over points on some other stratum $E(L_{\bullet})^{\circ}$. Since the stratum is not open, 
we have $k\geq 1$. We observe that for any $r$ such that $1\leq r \leq k$ the connected component of a fiber 
through $z$ lies in the fiber over $z_{r}\in\mathbb{P}(L_{r}^{\circ})$ of the projection 
\[
D(L_{r})=\mathbb{P}(L_{r}^{\sharp})\times\mathbb{P}((V/L_{r})^{\sharp})\rightarrow\mathbb{P}(L_{r}^{\sharp}) 
\quad \text{if} \ \kappa_{r}<1,
\]
or lies over $z_{r+1}\in\mathbb{P}((V/L_{r})^{\circ})$ of the projection 
\[
D(L_{r})=\mathbb{P}(L_{r}^{\sharp})\times\mathbb{P}((V/L_{r})^{\sharp})\rightarrow\mathbb{P}((V/L_{r})^{\sharp}) 
\quad \text{if} \ \kappa_{r}>1.
\]
(The latter one will be more clear when we deal with the elliptic case.) 
We notice that the fiber is smooth at any $z$ over that stratum $E(L_{\bullet})^{\circ}$, hence a connected component of 
that fiber is also irreducible. 
Then it is easy to see that every irreducible component in $D(L_{r})$ 
over $\{z_{r}\}\times\mathbb{P}((V/L_{r})^{\sharp})$ (when $\kappa_{r}<1$) 
or over $\mathbb{P}(L_{r}^{\sharp})\times\{z_{r+1}\}$ (when $\kappa_{r}>1$) is compact.
\end{proof}

We now establish the metric completion of $\mathbb{P}(G\backslash V^{\circ})$.

\begin{proposition}
In this case, the metric completion of $\mathbb{P}(G\backslash V^{\circ})$ is $\mathbb{P}(G\backslash V)$. 
Moreover, the stratum $\mathbb{P}(G\backslash L^{\circ})$ is totally geodesic in $\mathbb{P}(G\backslash V)$.
\end{proposition}

\begin{proof}
We first assume that $G$ is trivial. Since $\kappa_{0}=1$, the holonomy of the connection around the origin 
is a translation so that the developing map factors through a (multivalued) local isomorphism from $\mathbb{P}(V^{\circ})$ to 
an affine hyperplane, denoted by $A^{-}$, in $A=\mathbb{C}^{n+1}$. 
The translation generates a faithful action $\mathbb{C}^{+}$ on $A$ 
so that the space $A^{-}$ can be identified with $\mathbb{C}^{+}\backslash A$. 
By Theorem $\ref{thm:hermitian-forms}$ we know that the kernel of the Hermitian form 
is generated by the Euler field, so that $\mathbb{P}(V^{\circ})$ inherits a positive definite Hermitian form by restriction.

Also since $\kappa_{0}=1$, by the monotonicity of the exponents we have that $\kappa_{L}<1$ for all $L\in \mathscr{L}_{\mathrm{irr}}
(\mathscr{H})-\{0\}$. So the ``projective" developing map near a point of $D(L)^{\circ}$ is equivalent to the map with coordinates 
$[F_{0},t^{1-\kappa_{L}},t^{1-\kappa_{L}}F_{1}]=(\check{F_{0}},t^{1-\kappa_{L}},t^{1-\kappa_{L}}F_{1})$, 
which is essentially given by $\check{F_{0}}$. 
We note that here the projectivization is simply a restriction on an affine hyperplane in $A$, 
since the induced projective structure is just an affine structure taken from the affine hyperplane. 
So in this sense the map $\check{F_{0}}$ simply means the map $F_{0}$ with the component representing the 
direction of the Euler field being removed. For the same reason, therefore, the ``projective" developing map ($\ref{eqn:Pev-mapping}$) 
near a point of $E(L_{\bullet})^{\circ}$ is equivalent to the map with coordinates 
\[
\Big[(1,F_{1}),\big(t_{1}^{1-\kappa_{1}}\cdots t_{i-1}^{1-\kappa_{i-1}}(1,F_{i})\big)_{i=2}^{k+1}\Big]
=\Big(F_{1},\big(t_{1}^{1-\kappa_{1}}\cdots t_{i-1}^{1-\kappa_{i-1}}(1,F_{i})\big)_{i=2}^{k+1}\Big)
\]
which is essentially given by $F_{1}$.

This tells us that the extension across $E(L_{\bullet})^{\circ}$ by the developing map establishes a multivalued map 
from $\mathbb{P}(V^{\sharp})$ to the metrically complete space $A^{-}$. 
Observe that the space $A^{-}$ has a $\Gamma$-invariant positive definite Hermitian form: if the Hermitian form on $A$ is 
semipositive definite then it is clear; if the Hermitian form on $A$ is positive definite, then we can identify $A^{-}$ 
with the orthogonal complement of the translation of $A$. 
This extended map is certainly not injective near a point $z$ on $D(L)^{\circ}=\mathbb{P}(L^{\circ})\times\mathbb{P}((V/L)^{\circ})$ 
even modulo the action of the local monodromy group $\Gamma_{z}$, 
but by Lemma $\ref{lem:compact-fiber}$ we notice that 
the fiber over a point of $\mathbb{P}(L^{\circ})$ is compact. So that we can invoke the Stein factorization 
(Lemma $\ref{lem:Stein-factorization}$), for $\mathbb{P}(V^{\sharp})\rightarrow A^{-}$, to get an 
intermediate space, the Stein factor, denoted by $\mathbb{P}(V^{\natural})$, 
such that the differential of the extended map from $\mathbb{P}(V^{\natural})$ to $A^{-}$ at $z$ is injective with the 
transition maps (between charts around $z$) lying in $\Gamma_{z}$. 
This shows that the extended map is a local immersion up to $\Gamma$. 
Moreover, this factorization is realized by a 
projection $D(L)^{\circ}\rightarrow \mathbb{P}(L^{\circ})$, which means the transversal direction of $\mathbb{P}(L^{\sharp})$ for 
each $D(L)$ gets contracted. Since this is just the inverse operation for the blowup of each $L$, 
then we get back to the central exceptional divisor $\mathbb{P}(V)$, which gives the $\mathbb{P}(V^{\natural})$ for this case. 
It is clear that the stratum 
$E(L_{\bullet})^{\circ}$ determined by a flag $L_{\bullet}$ gets contracted to $\mathbb{P}(L_{1}^{\circ})$. 

Then $\mathbb{P}(V)$ is endowed with a (singular) flat K\"{a}hler metric pulled back from $A^{-}$ 
since the extended developing map is locally immersive up to $\Gamma$. 
On $\mathbb{P}(V^{\circ})$, it is clear that the inherited flat K\"{a}hler metric from $V^{\circ}$ agrees with the 
pulled back metric from $A^{-}$ since the developing map is a local isomorphism. 
Furthermore, we notice that this induced metric would not alter the original topology on $\mathbb{P}(V)$ 
(this is because the metric still varies over $\mathbb{P}(V)$ in a continuous manner) 
so that $\mathbb{P}(V)$ is compact and $\mathbb{P}(V^{\circ})$ is dense in $\mathbb{P}(V)$. 
This implies that $\mathbb{P}(V)$ is the metric completion of $\mathbb{P}(V^{\circ})$. 

We write the dimension of $L$ as $d_{L}$. 
Near a point of $\mathbb{P}(V^{\circ})$, its local curvature model is given by 
$\sqrt{-1}\partial\bar{\partial}(\sum_{i=1}^{n}|z_{i}|^{2})$ 
if we let $(z_{0},z_{1},\dots,z_{n})$ stand for an affine chart near a point of $V^{\circ}$ and 
let $z_{0}$ represent the direction of the Euler field. 
Now given a stratum $\mathbb{P}(L^{\circ})$, near a point of it, 
since the Hermitian form is flat for the Dunkl connection and $\Gamma$-invariant, 
we can substitute those coordinates in the curvature model by the new affine structure 
$(\check{F_{0}},t^{1-\kappa_{L}},t^{1-\kappa_{L}}F_{1})$ where $\check{F_{0}}$ can be understood as $(z_{1},\dots,z_{d_{L}-1})$, 
then we get a curvature model near the given point 
on $\mathbb{P}(L^{\circ})$ as $\sqrt{-1}\partial\bar{\partial}(\sum_{i=1}^{d_{L}-1}|z_{i}|^{2})$ by letting $t=0$. 
This shows that the stratum $\mathbb{P}(L^{\circ})$ is totally geodesic in $\mathbb{P}(V)$.

Now suppose the Schwarz symmetry group $G$ is not trivial, then it is actually generated by a so-called 
\emph{Schwarz rotation group} $G_{L}$ for each $L\in\mathscr{L}_{\mathrm{irr}}(\mathscr{H})$. 
The group $G_{L}$ is a finite subgroup of $U(V)$ fixing $L$ and acting on $L^{\perp}$ as scalar multiplication by 
$p_{L}$th roots of unity (if we write $p_{L}:=|G_{L}|$) such that the Dunkl system is invariant under the action of $G_{L}$. 
It implies that the extended developing map is a constant on the $G$-orbits. Therefore, the action of $G$ would not affect the 
metric completeness of $\mathbb{P}(V)$ or that each stratum being totally geodesic 
in the metric completion.

The proof of the proposition is now complete.
\end{proof}

\begin{proof}[Proof of Theorem $\ref{thm:parabolic}$]
It remains to show that the metric is of a conical type. We still first deal with the situation when $G$ is trivial. 
Note that for a flat metric the tangent cone at a given point just gives 
a local metric model near that point. This is because the exponential map is simply an identity map in this case. 
So for this case it is enough 
to describe the conical structure on the tangent cone in order to show that this singular metric is a cone metric. 

For that we now give an inductive description. For $L\in \mathscr{L}_{\mathrm{irr}}(\mathscr{H})$ of 
complex codimension $1$, namely, $L$ is just an $H$ in $\mathscr{H}$, near a point $z$ of $\mathbb{P}(H^{\circ})$, according to the 
constant zero curvature model ($\ref{eqn:zero}$), we have a curvature model 
$\sqrt{-1}\partial\bar{\partial}(\sum_{i=1}^{n-1}|z_{i}|^{2}+|z_{n}|^{2(1-\kappa_{H})})$ near that point, which shows that the tangent cone 
at that point is isomorphic to $\mathbb{C}^{n-1}\times \mathbb{C}_{1-\kappa_{H}}$. This is clearly a flat cone metric 
with cone angle $2\pi(1-\kappa_{H})$ along $\mathbb{P}(H^{\circ})$. 
In other words, it also shows that the tangent cone is just a flat metric cone over 
the unit tangent cone $S_{z}(\mathbb{P}(V))\cong S^{2n-3}*N_{z}(\mathbb{P}(V))$ such that the cone over $S^{2n-3}$ is totally geodesic 
and the normal cone $N_{z}(\mathbb{P}(V))$ is a spherical cone-manifold. 
For $L\in \mathscr{L}_{\mathrm{irr}}(\mathscr{H})$ of complex codimension $n+1-d_{L}$, 
the $\mathscr{H}_{L}$ defines a projective arrangement on the central exceptional divisor of $V/L$, i.e., $\mathbb{P}(V/L)$, 
we assume that $\mathbb{P}(V/L)$ is a complex $(n-d_{L})$-dimensional cone-manifold. For a point $z$ on $\mathbb{P}(L)^{\circ}$, its 
tangent cone is the direct product of a flat $\mathbb{C}^{d_{L}-1}$ and a singular $\mathbb{C}^{n+1-d_{L}}_{sing}$, 
where the latter stands for the singular structure on $V/L$ induced by $\mathscr{H}_{L}$. 
We know that the affine structure at $(V/L)_{0}$ is given by $(t^{1-\kappa_{L}},t^{1-\kappa_{L}}F_{1})=t^{1-\kappa_{L}}(1,F_{1})$. 
So the normal cone $N_{z}(\mathbb{P}(V))$ at $z$ is just the unit tangent cone of $V/L$ at $0$, i.e., $S_{0}(V/L)$, which is a 
principal circle bundle over $\mathbb{P}(V/L)$. This implies that for a point $x$ on $S_{0}(V/L)$, its unit tangent cone $S_{x}(S_{0}(V/L))$ 
is isomorphic to the join $S^{0}*S_{y}(\mathbb{P}(V/L))$ where $y$ is obtained by projecting $x$ to $\mathbb{P}(V/L)$. 
Hence it shows that the normal cone $N_{z}(\mathbb{P}(V))=S_{0}(V/L)$ is a prime spherical cone-manifold 
since the unit tangent cone at each point $x$ of $N_{z}(\mathbb{P}(V))$ is a spherical cone-manifold. 
So the tangent cone $T_{0}(V/L)$ of $V/L$ at $0$ is just a flat metric cone 
over $S_{0}(V/L)$ such that the length of each circle over a generic point of $\mathbb{P}(V/L)$ becomes $2\pi(1-\kappa_{L})$. 
Therefore, the tangent cone at $z$ of 
$\mathbb{P}(V)$ is a flat metric cone over the join $S^{2d_{L}-3}*N_{z}(\mathbb{P}(V))$ such that the cones over the strata 
in singular loci are totally geodesic. This shows that the metric near $z$ is locally 
isometric to a flat cone metric.

For a member $L\in \mathscr{L}(\mathscr{H})$ reducible, since a normal 
fiber of $\mathbb{P}(L^{\circ})$ in $\mathbb{P}(V)$ can be decomposed as a product of several $V/L'$'s for $L'\supset L$ irreducible, 
its tangent cone is just a direct product of a flat $\mathbb{C}^{d_{L}-1}$ and tangent cones at $0$ of $V/L'$ for each $L'$. 
Or equivalently its tangent cone is just a flat metric cone over the join of their unit tangent cones. This is certainly 
locally isometric to a flat cone metric since we already show that each unit tangent cone is a spherical cone-manifold 
for $L'$ irreducible. 
Therefore, we finish the inductive description which implies that $\mathbb{P}(V)$ is a parabolic (i.e., 
$(\mathbb{C}^{n},\mathrm{U}(n))-$) cone-manifold.

Now suppose the Schwarz symmetry group $G$ is not trivial, then we know that each Schwarz rotation group $G_{L}$ 
fixes $L$ and acts on $L^{\perp}$ as scalar multiplication by $p_{L}$th roots of unity such that the Dunkl system 
is invariant under the action of $G_{L}$. 
Hence the developing map near a point $z$ of $E(L)^{\circ}$ factors through $G_{L}\backslash V^{\sharp}_{z}$. 
Since $G_{L}$ acts on $V/L$ as (complex) scalar multiplication, it induces a free action on the unit tangent cone $S_{0}(V/L)$ 
of $V/L$ at $0$. So the unit tangent cone $S_{0}(G_{L}\backslash (V/L))$ is still a prime spherical cone-manifold. 
Then the tangent cone at $z$ of 
$\mathbb{P}(G\backslash V)$ is still a flat metric cone over the join $S^{2d_{L}-3}*N_{z}(\mathbb{P}(G\backslash V))$.  
Under this situation we only need to notice that for $L$ of complex codimension $1$, i.e., 
$L$ is a hyperplane $H\in\mathscr{H}$, the cone angle $2\pi(1-\kappa_{H})$ needs to be replaced by $2\pi(1-\kappa_{H})/p_{H}$ 
due to the action of $G_{H}$, since in this level $G_{H}$ acts as like the Euler field for $\nabla_{H}$ on $V/H$.  
The situation for $L$ reducible is easy to follow by applying the same argument as for $G$ trivial.

In addition, there might also exist some stratum in $\mathbb{P}(V)$ on which the action of $G$ is not trivial while 
$G$ acts on its preimage in $V$ trivially. This is due to the projectivization. 
We consider one of those strata in $\mathbb{P}(V)$, say $K$ for instance, so $G$ induces a subgroup which fixes $K$ and 
acts on the transversal direction of $K$ as a rotation. 
Then it still acts freely on the normal cone at a point of the stratum, 
so we can apply the same argument as above to $K$ to get the conical structure along $K$. 
Therefore, we can also conclude that the singular metric is still of a conical type near a point of the given stratum. 

This completes the proof.
\end{proof}

\begin{remark}
In the sense of Subsection $\ref{subsec:link}$, for $L\in \mathscr{L}_{\mathrm{irr}}(\mathscr{H})$ in this case, 
the space $G\backslash\mathbb{P}(V/L)$ is just the complex link of the stratum $\mathbb{P}(L^{\circ})$, and the space of  
real rays of $G\backslash (V/L)$ is just the real link of the stratum which is a Seifert fiber space over the 
complex link with generic fiber a circle having length $2\pi(1-\kappa_{L})$. This length is the 
scalar cone angle at $\mathbb{P}(L^{\circ})$. The number $1-\kappa_{L}$ is also the dilatation factor of the Euler field 
on the Dunkl system $V/L$ with the arrangement $\mathscr{H}_{L}$. These notions also apply to the following elliptic and hyperbolic cases.
\end{remark}

We now relate our $(\mathbb{C}^{n},\mathrm{U}(n))-$cone-manifold to the \emph{polyhedral K\"{a}hler manifold} introduced 
by Panov in \cite{Panov}. Recall from \cite[Section 1]{Panov} that a piecewise linear connected manifold $M$ with a fixed 
simplicial decomposition is called a \emph{polyhedral manifold} if we can choose a flat metric on every simplex of highest 
dimension in such a way that every two simplices having a common face are glued by an isometry. Note that the existence of such a 
polyhedral metric is independent of the choice of the simplicial decomposition. 
Then a polyhedral manifold $M$ of real dimension $2n$ is called a \emph{polyhedral K\"{a}hler manifold} if the holonomy of its 
metric belongs to a subgroup of $\mathrm{SO}(2n)$ conjugate to $\mathrm{U}(n)$ and every real codimension $2$ face with 
cone angle $2m\pi$ ($m\geq 2,m\in\mathbb{N}$) has a holomorphic direction.

\begin{theorem}
A complete $(\mathbb{C}^{n},\mathrm{U}(n))-$cone-manifold $M$, assuming the cone angles along real codimension 2 strata are 
$\leq 2\pi$, is a polyhedral K\"{a}hler manifold.
\end{theorem}

\begin{proof}
Let us verify the conditions for $M$ to be a polyhedral K\"{a}hler manifold. The cone-manifold $M$ is modeled on $X=\mathbb{C}^{n}$, 
then it is clear that for a fixed simplicial decomposition of $M$ there exists a flat metric on every simplex of dimension $2n$ 
such that every two simplices with a common face are glued by an isometry. So $M$ is a polyhedral manifold. Since we restrict the 
isometry group $G$ to be the unitary group $\mathrm{U}(n)$ for our cone-manifold $M$, 
the holonomy of its metric clearly belongs to $\mathrm{U}(n)$. 
Given the cone angles along real codimension $2$ strata are assumed to be at most $2\pi$ (so that the complex structure in 
real codimension $2$ strata is automatic), hence, we have the assertion.
\end{proof}

\begin{remark}
Therefore, our parabolic case gives an affirmative answer to a conjecture raised by Panov on Page 2207 of \cite{Panov}: 
the metrics in the curvature zero case of the Dunkl system are polyhedral K\"{a}hler. 
That is to say, our parabolic $\mathbb{P}(G\backslash V)$ provides examples of the so-called polyhedral K\"{a}hler manifolds.
\end{remark}

\subsection{Elliptic case}

Next we treat the elliptic case. The main result for the elliptic case is as follows. 

\begin{theorem}\label{thm:elliptic}
	Let be given a Dunkl system. Suppose that $\kappa_{H}\in (0,+\infty)$ for every $H\in \mathscr{H}$, 
	that there is a flat positive definite Hermitian form on the tangent bundle of $V^{\circ}$. Assume that 
	there is no $L\in\mathscr{L}_{\mathrm{irr}}(\mathscr{H})$ with $\kappa_{L}=1$ and that the intersection of any two members 
	of $\mathscr{L}_{\mathrm{irr}}(\mathscr{H})$ on which $\kappa>1$ is still irreducible. 
	Then the projective developing map realizes the space $\mathbb{P}(G\backslash V^{\circ})$ as a divisor complement of 
	its metric completion, via which $\mathbb{P}(G\backslash V^{\circ})$ acquires the structure of an elliptic cone-manifold 
	(i.e., $(\mathbb{CP}^{n},\mathrm{PU}(n+1))-$cone-manifold). 
	In this case, the metric completion is $\mathbb{P}(G\backslash V^{\natural})$ (which we will describe below).
\end{theorem}

Now besides $\kappa_{L}<1$ it is also possible that $\kappa_{L}> 1$ for some $L$. 
So let us have a look how the projective developing map behaves around a stratum $\mathbb{P}(L^{\circ})$ for which $\kappa_{L}>1$. 
For that we still need the big resolved space $\mathbb{P}(V^{\sharp})$ 
which is obtained in the same way as described before. 
It is already known that the affine structure on $V^{\circ}$ degenerates infinitesimally simply along $E(L)^{\circ}$ with 
logarithmic exponent $\kappa_{L}-1$. If it is $>0$, the associated affine foliation of $E(L)^{\circ}$ is given by its projection 
onto $\mathbb{P}((V/L)^{\circ})$. Since we have a flat positive definite Hermitian form in this case, the restrictive `infinitesimally' 
can be dropped, i.e., $\kappa_{L}-1\notin\mathbb{Z}_{>0}$.

Then near a point $z\in D(L)^{\circ}$, when $\kappa_{L}>1$, the 
projective developing map is equivalent to the map 
\[
[F_{0},t^{1-\kappa_{L}},t^{1-\kappa_{L}}F_{1}]=[t^{\kappa_{L}-1}F_{0},1,F_{1}]
\]
which is essentially given by $F_{1}$. That means the extension of the map across $D(L)^{\circ}$ is realized by the projection 
$D(L)^{\circ}\rightarrow \mathbb{P}((V/L)^{\circ})$. In other words, the own direction of $\mathbb{P}(L)$ would be contracted 
so as to get a space compatible with the projective structure.

We still first establish the metric completion of $\mathbb{P}(G\backslash V^{\circ})$ for this case.

\begin{proposition}
In this case, the metric completion of $\mathbb{P}(G\backslash V^{\circ})$ is $\mathbb{P}(G\backslash V^{\natural})$, 
the Stein factor of the (multivalued) map $\mathbb{P}(G\backslash V^{\sharp})\rightarrow\mathbb{P}(A)$. 
Moreover, the stratum obtained after the birational operations induced by the projective structure 
is totally geodesic in $\mathbb{P}(G\backslash V^{\natural})$.
\end{proposition}

\begin{proof}
As before, we first assume that $G$ is trivial. To see the metric completion we need to analyze the behavior of 
the projective developing map near any stratum of $\mathbb{P}(V^{\sharp})$. We know that a stratum of $\mathbb{P}(V^{\sharp})$ is 
determined by a flag $L_{\bullet}:L_{0}=\{0\}\subset L_{1}
\subset\dots \subset L_{k}\subset V=L_{k+1}$. By the monotonicity of the exponents, since no $L\in\mathscr{L}_{\mathrm{irr}}(\mathscr{H})$ 
has the property that $\kappa_{L}=1$, we may assume that $1-\kappa_{1}<\cdots <1-\kappa_{r-1}<0<1-\kappa_{r}<\cdots <1-\kappa_{k}$, 
then near a point in $E(L_{\bullet})^{\circ}$, the projective developing map ($\ref{eqn:Pev-mapping}$) is equivalent to  
\begin{multline*}
	\Big[\big(t_{1}^{1-\kappa_{1}}\cdots t_{i-1}^{1-\kappa_{i-1}}(1,F_{i})\big)_{i=1}^{k+1}\Big]=\\
	=\Big[\big(t_{i}^{\kappa_{i}-1}\cdots t_{r-1}^{\kappa_{r-1}-1}(1,F_{i})\big)_{i=1}^{r-1},(1,F_{r}),\big(t_{r}^{1-\kappa_{r}}
	\cdots t_{i-1}^{1-\kappa_{i-1}}(1,F_{i})\big)_{i=r+1}^{k+1}\Big]
\end{multline*}
which is essentially given by $F_{r}$. That means the extension of the map across $E(L_{\bullet})^{\circ}$ is realized by the projection 
$E(L_{\bullet})^{\circ}\rightarrow \mathbb{P}((L_{r}/L_{r-1})^{\circ})$.

This tells us that the extension across $E(L_{\bullet})^{\circ}$ by the projective developing map establishes a multivalued map 
from $\mathbb{P}(V^{\sharp})$ to $\mathbb{P}(A)$. 
Observe that the space $A$ has a $\Gamma$-invariant positive definite Hermitian form so that $\mathbb{P}(A)$ inherits a 
Fubini-Study metric from the complete space $A$. 
This extended map is certainly not injective near a point $z$ on $E(L_{\bullet})^{\circ}$ even up to the local monodromy group 
$\Gamma_{z}$, but by Lemma $\ref{lem:compact-fiber}$ we notice that 
the fiber over a point of $E(L_{\bullet})^{\circ}$ is compact. So that we can invoke the Stein factorization 
(Lemma $\ref{lem:Stein-factorization}$) again, for $\mathbb{P}(V^{\sharp})\rightarrow \mathbb{P}(A)$, to get an 
intermediate space, i.e., the Stein factor, denoted by $\mathbb{P}(V^{\natural})$, 
such that the differential of the extended map from the Stein factor to $\mathbb{P}(A)$ at $z$ is injective with transition maps 
(between charts around $z$) lying in $\Gamma_{z}$, and hence the extended map is a local immersion up to $\Gamma$. 
This factorization is realized by a projection $E(L_{\bullet})^{\circ}\rightarrow \mathbb{P}((L_{r}/L_{r-1})^{\circ})$. 
This can be obtained by contracting the transversal direction of $\mathbb{P}(L^{\sharp})$ for each 
$D(L)=\mathbb{P}(L^{\sharp})\times \mathbb{P}((V/L)^{\sharp})$ if $\kappa_{L}<1$, and contracting 
its own direction of $\mathbb{P}(L^{\sharp})$ for each $D(L)=\mathbb{P}(L^{\sharp})\times \mathbb{P}((V/L)^{\sharp})$ 
if $\kappa_{L}>1$. 

Then $\mathbb{P}(V^{\natural})$ is endowed with a (singular) Fubini-Study metric pulled back from $\mathbb{P}(A)$ 
since the resulting extended projective developing map is locally immersive up to $\Gamma$. 
It is clear that on $\mathbb{P}(V^{\circ})$ it is endowed with a smooth Fubini-Study metric by the given positive definite 
Hermitian form on $V^{\circ}$, which agrees with the metric pulled back from $\mathbb{P}(A)$ by the projective developing map. 
We further notice that this induced metric would not alter the original topology on $\mathbb{P}(V^{\natural})$ 
(since the metric varies over $\mathbb{P}(V^{\natural})$ continuously) 
so that $\mathbb{P}(V^{\natural})$ is compact  
and $\mathbb{P}(V^{\circ})$ is clearly dense in $\mathbb{P}(V^{\natural})$. 
This implies that $\mathbb{P}(V^{\natural})$ is the metric completion of $\mathbb{P}(V^{\circ})$. 

The given positive definite Hermitian form endows $\mathbb{P}(V^{\circ})$ with a K\"{a}hler metric of constant positive holomorphic 
sectional curvature. Since the Hermitian form is flat with respect to the Dunkl connection and $\Gamma$-invariant, 
for $D(L)^{\circ}$ we still substitute the degenerate projective structure into the 
local smooth model $\sqrt{-1}\partial\bar{\partial}\log (1+\sum_{i=1}^{n}|z_{i}|^{2})$. 
Take the situation near a point of 
$\mathbb{P}((V/L)^{\circ})$ for which $\kappa_{L}>1$ for instance, the projective developing map can be written as 
$[t^{\kappa_{L}-1}F_{0},1,F_{1}]$, and hence the curvature model near the point on $\mathbb{P}((V/L)^{\circ})$ is given by 
$
\sqrt{-1}\partial\bar{\partial}\log \big(1+\sum_{i=d_{L}+1}^{n}|z_{i}|^{2}\big)
$
by letting $t=0$. Likewise, the curvature models around other strata can be obtained in the same way as long as 
the projective structure is of the same shape.
Hence the curvature models for other strata are given as 
$
\sqrt{-1}\partial\bar{\partial}\log \big(1+\sum_{i=1}^{d_{L}-1}|z_{i}|^{2}\big)
$
on $\mathbb{P}(L^{\circ})$ for which $\kappa_{L}<1$; and 
$
\sqrt{-1}\partial\bar{\partial}\log \big(1+\sum_{i=d_{L'}+1}^{d_{L}-1}|z_{i}|^{2}\big)
$
on $\mathbb{P}((L/L')^{\circ})$ for which $L'\subset L$ with $\kappa_{L'}>1>\kappa_{L}$. (But note that 
$\mathbb{P}((L/L')^{\circ})$ is not irreducible in the sense that its transversal direction in $\mathbb{P}(V^{\natural})$ can be 
decomposed as $\mathbb{P}(L')\times\mathbb{P}(V/L)$.) 
This shows that all the above strata are totally geodesic in $\mathbb{P}(V^{\natural})$.

Now let us assume that $G$ is not trivial, as in the parabolic case, $G$ is a finite group generated by the Schwarz rotation 
group $G_{L}$ for each $L\in\mathscr{L}_{\mathrm{irr}}(\mathscr{H})$. 
The group $G_{L}$ is a finite subgroup of the projective unitary group of $V^{\natural}$ 
fixing $L$ (resp. $V/L$) and acting on $L^{\perp}$ (resp. $(V/L)^{\perp}$) as scalar multiplication by $p_{L}$th ($p_{L}:=|G_{L}|$) 
roots of unity when $\kappa_{L}<1$ (resp. $\kappa_{L}>1$) 
such that the Dunkl system is invariant under the action of $G_{L}$. 
It implies that the extended developing map is a constant on the $G$-orbits. 
Therefore, the action of $G$ would not affect the metric completion of $\mathbb{P}(V^{\natural})$ or that 
the newly obtained stratum being totally geodesic in $\mathbb{P}(V^{\natural})$.

The proof of the proposition is now complete.
\end{proof}

\begin{proof}[Proof of Theorem $\ref{thm:elliptic}$]
For this case it remains to show that the metric on $\mathbb{P}(G\backslash V^{\natural})$ is of a conical type. 
We first deal with the situation when $G$ is trivial. We still do it in an inductive way. 
For $L\in \mathscr{L}_{\mathrm{irr}}(\mathscr{H})$ 
whose $\kappa_{L}<1$, the unit tangent cone at a point of $\mathbb{P}(L^{\circ})$ can be obtained in the same way as in the 
parabolic case. So we deal with the strata for which $\kappa_{L}>1$. For $L$ of dimension $1$, according to the local model 
($\ref{eqn:positive}$), the curvature model near a point $z$ of $\mathbb{P}((V/L)^{\circ})$ on $\mathbb{P}(V^{\natural})$ is given as 
$
\sqrt{-1}\partial\bar{\partial}\log \big(1+|z_{1}|^{2(\kappa_{L}-1)}+\sum_{i=2}^{n}|z_{i}|^{2}\big)
$, 
so it is clearly a Fubini-Study metric with cone angle $2\pi(\kappa_{L}-1)$ along $\mathbb{P}((V/L)^{\circ})$. 
Its unit tangent cone is the join $S^{2n-3}*N_{z}(\mathbb{P}(V^{\natural}))$ which is a spherical cone-manifold. 
For $L$ of higher dimension, the $\mathscr{H}^{L}$ defines an arrangement on $L$, the degenerate affine structure near $(L)_{0}$ 
($L$ regarded as the fiber to $\mathbb{P}((V/L)^{\circ})$) 
can be written as $t^{\kappa_{L}-1}F_{0}=(t^{\kappa_{L}-1},t^{\kappa_{L}-1}f_{0})$ where $f_{0}$ is an affine chart 
for $\mathbb{P}(L)$, 
assuming the cone-manifold structure on $\mathbb{P}(L)$, by the same argument for $L$ whose $\kappa_{L}<1$ as in the parabolic case, 
the normal cone $N_{z}(\mathbb{P}(V^{\natural}))$ at a point $z$ of $\mathbb{P}((V/L)^{\circ})$ is a spherical cone-manifold 
so that the unit tangent cone $S_{z}(\mathbb{P}(V^{\natural}))$ at $z$ is isomorphic to a join 
$S^{2n-2d_{L}-1}*N_{z}(\mathbb{P}(V^{\natural}))$ whose strata in its singular loci are always of real odd dimension. 
For a stratum in $\mathbb{P}(V^{\natural})$ not irreducible, its normal fiber can be decomposed as a product of several normal fibers of 
irreducible strata in  $\mathbb{P}(V^{\natural})$, so its normal cone is just a join of normal cones of those irreducible strata. 
Therefore, the local metric model near a point $z$ of $\mathbb{P}(V^{\natural})$ can be simply given by a Fubini-Study metric 
cone $C_{g}(S_{z}(\mathbb{P}(V^{\natural})))$ over $S_{z}(\mathbb{P}(V^{\natural}))$, 
where $g$ is the endowed Fubini-Study metric, such that the cones over those strata in 
singular loci in $S_{z}(\mathbb{P}(V^{\natural}))$ are totally geodesic in $C_{g}(S_{z}(\mathbb{P}(V^{\natural})))$. This shows that 
$\mathbb{P}(V^{\natural})$ is an elliptic (i.e., $(\mathbb{CP}^{n},\mathrm{PU}(n+1))-$) cone-manifold.

Now we assume $G$ is not trivial, 
as in the parabolic case, we note that the developing map near a point $z$ of $E(L)^{\circ}$ 
factors through $G_{L}\backslash V^{\sharp}_{z}$. 
Then we can also apply the inductive argument to 
$G\backslash L$ if $\kappa_{L}>1$ (in this case, the induction starts from dimension $1$, note the cone angle replaced by 
$2\pi(\kappa_{L}-1)/p_{L}$ when $\dim_{\mathbb{C}}L=1$), 
or $G\backslash (V/L)$ if $\kappa_{L}<1$, 
so as to get the conical structure along the stratum $\mathbb{P}(G\backslash(V/L)^{\circ})$ or $\mathbb{P}(G\backslash L^{\circ})$. 

In addition, we also notice that there might exist some stratum, $K$ say, in $\mathbb{P}(V^{\natural})$ fixed by some subgroup 
of $G$ while which acts trivially on its preimage in $V^{\natural}$. 
Since it acts as a rotation on the normal fiber of $K$ which means the action on the normal cone is free, 
it would certainly maintain the conical structure along that stratum. 
Therefore, the conclusion follows as well by applying the same inductive argument to the normal fiber of $G\backslash K$.

This completes the proof.
\end{proof}

\begin{remark}
If the Dunkl system admits a flat positive definite Hermitian form, it is likely that the exponent $\kappa_{L}$ of an 
irreducible member $L\in\mathscr{L}_{\mathrm{irr}}(\mathscr{H})$ would never take the value $1$.
\end{remark}

For many situations in the elliptic case we have that $\kappa_{0}<1$, so we have a direct corollary from the above theorem.

\begin{theorem}
Let be given a Dunkl system. Suppose that $\kappa_{H}\in (0,+\infty)$ for every $H\in \mathscr{H}$, 
that there is a flat positive definite Hermitian form on the tangent bundle of $V^{\circ}$ and that $\kappa_{0}<1$. 
Then the projective developing map realizes the space $\mathbb{P}(G\backslash V^{\circ})$ as a divisor complement of 
its metric completion, via which $\mathbb{P}(G\backslash V^{\circ})$ acquires the structure of an elliptic cone-manifold. 
In this case, the metric completion is $\mathbb{P}(G\backslash V)$.
\end{theorem}

\begin{proof}
Since $\kappa_{0}<1$, we have $\kappa_{L}<1$ for any $L\in\mathscr{L}_{\mathrm{irr}}(\mathscr{H})$ by the monotonicity of the 
exponents. Then this is a subcase of the above theorem. In this case, the metric completion $\mathbb{P}(G\backslash V^{\natural})$ 
is just $\mathbb{P}(G\backslash V)$.
\end{proof}

\subsection{Hyperbolic case}

In this subsection we treat the hyperbolic case which needs somewhat more effort (dealing with the situation when $\kappa_{L}=1$). 

Now the affine space $A$ in which the developing map takes values is in fact a vector space (it has an origin) 
equipped with a nondegenerate Hermitian form $h$ of hyperbolic type. 
We denote by $A_{\mathbb{B}}$ the set of vectors in $A$ on which the Hermitian form takes negative self-product and 
by $\mathbb{B}$ its projectivization in $\mathbb{P}(A)$. Notice that $\mathbb{B}$ can be identified with the 
complex hyperbolic space $\mathbb{CH}^{n}$ and $A_{\mathbb{B}}$ can be thought of as a $\mathbb{C}^{\times}$-bundle over $\mathbb{B}$. 
The admissibility condition simply means that the developing map takes its values in $A_{\mathbb{B}}$ and hence 
the projective developing map takes its values in $\mathbb{B}$.

The main result for the hyperbolic case is as follows.

\begin{theorem}\label{thm:hyperbolic}
Let be given a Dunkl system. Suppose that $\kappa_{H}\in (0,+\infty)$ for every $H\in \mathscr{H}$, 
that it admits a flat admissible Hermitian form of hyperbolic type. 
Then the projective developing map realizes the space $\mathbb{P}(G\backslash V^{\circ})$ as a divisor complement of 
its metric completion, via which $\mathbb{P}(G\backslash V^{\circ})$ acquires the structure of a hyperbolic cone-manifold 
(i.e., $(\mathbb{CH}^{n},\mathrm{PU}(n,1))-$cone-manifold) of finite volume. In this case, the metric completion is 
$\mathbb{P}(G\backslash V^{\natural})$ (which we will describe below).
\end{theorem}

Our resolved space $V^{\sharp}$ for this case is obtained by blowing up the members of $\mathscr{L}_{\mathrm{irr}}(\mathscr{H})$ 
in their natural partial order (so starting from the smallest dimension), 
among which we blow up each $L$ with $\kappa_{L}=1$ in a \emph{real-oriented manner}. 
Notice that the action of the group $G$ can be extended to $V^{\sharp}$ naturally as well. 
It is then clear that $V^{\sharp}$ is a manifold with boundary $\partial V^{\sharp}$. Its interior $V^{\sharp}-\partial V^{\sharp}$ 
is a quasiprojective variety and contains $V^{\circ}$ as an open-dense subset. 
Now the complement of $V^{\circ}$ in $V^{\sharp}-\partial V^{\sharp}$ is a normal crossing divisor whose 
closure intersects the boundary wall transversally.

When $\kappa_{L}>1$, as in the elliptic case, the associated affine foliation of $E(L)^{\circ}$ is given by its projection 
onto $\mathbb{P}((V/L)^{\circ})$, then by Lemma 6.7(i) of \cite{Couwenberg-Heckman-Looijenga} the restriction of the given 
Hermitian form to the fibers of the retraction $V_{L^{\circ}}\rightarrow (V/L)^{\circ}$ (i.e., the fibers of the retraction 
$E(L)^{\circ}\rightarrow \mathbb{P}(V/L)^{\circ}$) is positive, so that the restrictive `infinitesimally' 
can also be dropped, i.e., $\kappa_{L}-1\notin\mathbb{Z}_{>0}$.

Then let us look at the situation when $\kappa_{L}=1$. For that we need to introduce the \emph{Borel-Serre extension} 
associated to $\Gamma$, which could be viewed as a partial compactification of $\mathbb{B}$. 
Let us briefly recall the construction of the Borel-Serre extension over here, 
readers can refer to Section 6.4 of \cite{Couwenberg-Heckman-Looijenga} or Section 4.2 of \cite{Looijenga} 
for a more detailed description. 
Let $I\subset A$ be an isotropic line. Then the unitary transformations which respect the flag 
$\{0\}\subset I\subset I^{\perp}\subset A$ and act trivially on the successive quotients form a so-called 
\emph{Heisenberg group} $N_{I}$. Notice that the one-dimensional complex vector space $I\otimes\bar{I}$ has a natural 
real structure whose orientation is given by the positive ray of the elements $e\otimes e$, where $e$ runs over the 
generators of $I$. The center $Z(N_{I})$ of the Heisenberg group can be parametrized by the real line 
$\sqrt{-1}I\otimes\bar{I}(\mathbb{R})$ and the quotient $N_{I}/Z(N_{I})$ hence can be identified with the 
vector group $I^{\perp}/I\otimes\bar{I}$.

Note that the orbits of the positive ray of elements in $I\otimes\bar{I}$ are (oriented) geodesic rays in $\mathbb{B}$ 
which point to $[I]\in\partial\mathbb{B}$. If we fix a generator $e$ of $I$, then $\mathbb{B}$ can be realized as a 
subset of a hyperplane in $A$ defined by $h(z,e)=1$, which is known as the realization of $\mathbb{B}$ as a Siegel 
domain of the second kind. Under this realization the geodesic ray action becomes a group of translations over 
negative multiples of $e$. We write the space of these rays as $\mathbb{B}(I)$ so that we have a projection 
$\pi_{I}:\mathbb{B}\rightarrow\mathbb{B}(I)$ with fibers being these rays. We consider the disjoint union 
$\mathbb{B}\sqcup\mathbb{B}(I)$ endowed with the topology generated by (i) the open subsets of $\mathbb{B}$ and 
(ii) the subsets of the form $U\sqcup\pi_{I}(U)$ where $U$ runs over the open subsets of $\mathbb{B}$ invariant 
under $N_{I}$ and the positive ray in $I\otimes\bar{I}$. We call this topology the \emph{Borel-Serre topology}.

Now for each isotropic line $I\subset A$ such that $\Gamma\cap N_{I}$ is cocompact we add the partial boundary 
$\mathbb{B}(I)$ to $\mathbb{B}$ endowed with the above topology so that we get a manifold with boundary, 
denoted by $\mathbb{B}^{\sharp}$. 
We call this $\mathbb{B}^{\sharp}$ the Borel-Serre extension associated to $\Gamma$. There are infinitely many 
connected components in its boundary when $\Gamma$ does not act on $\mathbb{B}$ cocompactly. But it is worthwhile to 
note that the action of $\Gamma$ on the boundary is cocompact.

Each $L\in\mathscr{L}_{\mathrm{irr}}(\mathscr{H})$ with $\kappa_{L}\neq 1$ defines a divisor $E(L)$ in $V^{\sharp}$ as before. 
While any other $L\in\mathscr{L}_{\mathrm{irr}}(\mathscr{H})$ with $\kappa_{L}=1$ defines a boundary component $\partial_{L}V^{\sharp}$ 
and it contains a unique open dense stratum which can be identified with the product $L^{\circ}\times\mathbb{S}((V/L)^{\circ})$, 
where $\mathbb{S}$ assigns to a real vector space the sphere of its real half lines. Therefore, 
the developing map near $(\partial_{L}V^{\sharp})^{\circ}$ is equivalent to 
\[
(F_{0},\log t,F_{1})
\]
which shows that the map escapes to infinity when it approaches to $L^{\circ}$ along a curve in $V/L$. 

These divisors and boundary components intersect normally in an evident manner so that we have a natural stratification of $V^{\sharp}$. 
Then an arbitrary stratum is defined by a collection of divisors and boundary components determined by a subset 
of $\mathscr{L}_{\mathrm{irr}}(\mathscr{H})$: it has a nonempty intersection if and only if that subset makes up a 
flag $L_{\bullet}:L_{0}\subset\cdots\subset L_{r}\subset L_{r+1}\subset\cdots\subset L_{k}\subset L_{k+1}=V$. 
In that case their intersection contains an open dense stratum $S(L_{\bullet})$ which decomposes as 
\[
S(L_{\bullet})=L_{0}^{\circ}\times\prod_{i=1}^{r}\mathbb{P}((L_{i}/L_{i-1})^{\circ})\times\mathbb{P}((L_{r+1}/L_{r})^{\circ}) 
\times\prod_{i=r+2}^{k+1}\mathbb{P}((L_{i}/L_{i-1})^{\circ})
\]
if there is no $L_{i}$ whose exponent is one; but if the exponent of some $L_{i}$, say $L_{r}$, happens to be 
one, then the factor $\mathbb{P}((L_{r+1}/L_{r})^{\circ})$ should be replaced by $\mathbb{S}((L_{r+1}/L_{r})^{\circ})$.

The preimage $\mathbb{P}(V^{\sharp})$ of the origin of $V$ in $V^{\sharp}$ is a compact manifold with boundary 
$\mathbb{P}(\partial V^{\sharp})$. So the manifold interior $\mathbb{P}(V^{\sharp}-\partial V^{\sharp})$ is a 
quasiprojective manifold which contains $\mathbb{P}(V^{\circ})$ as the complement of a normal crossing divisor. 
The strata in $\mathbb{P}(V^{\sharp})$ are determined by the flag $L_{\bullet}$ as described above provided that $L_{0}=\{0\}$.

We first collect some properties of the extended projective developing map.

\begin{proposition}\label{prop:hyp-compact-fiber}
The projective developing map extends to $\mathbb{P}(V^{\sharp})\rightarrow \mathbb{B}^{\sharp}$ with range in the 
Borel-Serre extension as a (multivalued) continuous map 
which is a constant on the $G$-orbits. Then we have: 
\begin{enumerate}[label=(\roman*)]
\item Up to $\Gamma$, it sends every boundary component of $\mathbb{P}(V^{\sharp})$ to a Borel-Serre boundary component 
of $\mathbb{B}^{\sharp}$ 
and the restriction to the interior $\mathbb{P}(V^{\sharp}-\partial V^{\sharp})\rightarrow\mathbb{B}$ is a holomorphic map.

\item Every connected component of a fiber of the map $\mathbb{P}(V^{\sharp})\rightarrow \mathbb{B}^{\sharp}$ is compact.
\end{enumerate}
\end{proposition}

\begin{proof}
The proof requires us to analyze the behavior of the projective 
developing map near any stratum of $\mathbb{P}(V^{\sharp})$. Since a stratum of $\mathbb{P}(V^{\sharp})$ is determined by a flag 
$L_{\bullet}$, we divide its formation into 2 cases: 1) there is no $L$ such that $\kappa_{L}=1$ in the flag $L_{\bullet}$ and, 2) 
there exists an $L$ such that $\kappa_{L}=1$.

The first case, i.e., the case without boundary, can be dealt with in the same way as in the elliptic case. 
So let us consider the second case: there exists an $L_{i}\in L_{\bullet}$, say $L_{r}$, such that its exponent $\kappa_{r}=1$, 
i.e., the case with boundary components. Then the stratum $S(L_{\bullet})$ determined by the flag 
$L_{\bullet}:\{0\}=L_{0}\subset\cdots\subset L_{r}\subset L_{r+1}\subset\cdots\subset L_{k}\subset L_{k+1}=V$
decomposes as 
\[
S(L_{\bullet})=\prod_{i=1}^{r}\mathbb{P}((L_{i}/L_{i-1})^{\circ})\times\mathbb{S}((L_{r+1}/L_{r})^{\circ}) 
\times\prod_{i=r+2}^{k+1}\mathbb{P}((L_{i}/L_{i-1})^{\circ}),
\]
and the developing map over a point on that stratum is affine-linearly equivalent to a multivalued map taking values 
in $\mathbb{C}\times T_{1}\times\cdots\times\mathbb{C}\times T_{r}\times\mathbb{C}\times T_{r+1}\times\cdots\times\mathbb{C}
\times T_{k+1}$ as follows: 
\begin{multline*}
\Big(\big(t_{0}^{1-\kappa_{0}}\cdots t_{i-1}^{1-\kappa_{i-1}}(1,F_{i})\big)_{i=1}^{r-1},
t_{0}^{1-\kappa_{0}}\cdots t_{r-1}^{1-\kappa_{r-1}}(1,F_{r},\log t_{r},F_{r+1}),\\
\big(t_{0}^{1-\kappa_{0}}\cdots t_{i-1}^{1-\kappa_{i-1}}(1,F_{i})\big)_{i=r+2}^{k+1}\Big)
\end{multline*}
assuming that $1-\kappa_{0}<1-\kappa_{1}<\cdots <1-\kappa_{r-1}<0=1-\kappa_{r}<1-\kappa_{r+1}<\cdots <1-\kappa_{k}$.

Here we note that for the real-oriented blowup, $\log t_{r}$ could be viewed as a coordinate: its imaginary part 
$\arg t_{r}$ is used to parametrize the ray space $\mathbb{S}((L_{r+1}/L_{r})^{\circ})$ and its real part $\log |t_{r}|$ 
should be allowed to take the boundary value $-\infty$. We thus rewrite this coordinate as $\tau_{r}$. Then on a connected 
component of $S(L_{\bullet})$ in $\mathbb{P}(V^{\sharp})$, we have that $(F_{1},t_{1},\dots,\tau_{r},\dots,t_{k},F_{k+1})$ 
constitutes a chart for $\mathbb{P}(V^{\sharp})_{z}$ and the projective developing map becomes 
\[
\Big[\big(t_{i}^{1-\kappa_{i}}\cdots t_{r-1}^{1-\kappa_{r-1}}(1,F_{i})\big)_{i=1}^{r-1},
(1,F_{r},\tau_{r},F_{r+1}),
\big(t_{r+1}^{1-\kappa_{r+1}}\cdots t_{i-1}^{1-\kappa_{i-1}}(1,F_{i})\big)_{i=r+2}^{k+1}\Big]
\]
which is essentially given by $(F_{r},\tau_{r},F_{r+1})$.

For (i), we use the above explicit descriptions to see the projective developing map extends to 
$\mathbb{P}(V^{\sharp})\rightarrow \mathbb{B}^{\sharp}$. The constant component $1$ ahead of $F_{r}$ shows that 
the map sends the stratum to an affine chart of a projective space. While from before we already know that there is 
a realization of $\mathbb{B}^{\sharp}$ such that a chart of $\mathbb{B}^{\sharp}$ is given by the affine plane in $A$ 
defined by $h(-,e)=1$, where $e$ is minus the unit vector corresponding to the slot occupied by $\tau_{r}$. 
Under this realization the geodesic ray action is then given by the translations over negative multiples of $e$. 
So we have in fact a chart for the Borel-Serre compactification, as long as we allow $\tau_{r}$ to take its values 
on $[-\infty,+\infty)+\sqrt{-1}\mathbb{R}$. It follows that the projective developing map extends to 
$\mathbb{P}(V^{\sharp})\rightarrow \mathbb{B}^{\sharp}$. In particular, it sends the boundary stratum $S(L_{\bullet})$ 
to the Borel-Serre boundary (for $\mathrm{Re}(\tau_{r})$ taking value $-\infty$). It is also clear, by the case without 
boundary, that the restriction on the interior $\mathbb{P}(V^{\sharp}-\partial V^{\sharp})\rightarrow\mathbb{B}$ 
is holomorphic, since those component morphisms $F_{i}$ are all holomorphic. 
But these are all of only local nature because it is up to the action of $\Gamma$ on the range $\mathbb{B}$.

For (ii), by Lemma $\ref{lem:compact-fiber}$ the statement holds for the case without boundary. So we only need to deal with 
the case with boundary. For a point $z$ on any other stratum $S(L_{\bullet})$, we have $k\geq 1$ since the stratum is not 
open anymore. Without loss of generality we assume that $1-\kappa_{k-1}<1-\kappa_{k}=0$ so that we need only to consider 
contracting the longitudinal direction, the fiber passing through $z$ is locally given by $(F_{k},\tau_{k},F_{k+1})$ where 
$\mathrm{Re}(\tau_{k})$ takes value $-\infty$. We see that the fiber is smooth at any $z$ over that stratum $S(L_{\bullet})$, 
hence a connected component of that fiber is also irreducible. So an irreducible component of a fiber 
through $z$ lies in the fiber over $(z_{k},z_{k+1})\in\mathbb{P}((L_{k}/L_{k-1})^{\circ})\times\mathbb{S}((V/L_{k})^{\circ})$ 
of the projection 
\begin{multline*}
\partial \mathbb{P}(V^{\sharp})\cap D(L_{k-1})=\mathbb{P}(L_{k-1}^{\sharp})\times\mathbb{P}((L_{k}/L_{k-1})^{\sharp})
\times\mathbb{S}((V/L_{k})^{\sharp})\\
\rightarrow \mathbb{P}((L_{k}/L_{k-1})^{\sharp})\times\mathbb{S}((V/L_{k})^{\sharp}),
\end{multline*}
where $z_{k+1}$ should be understood as a real coordinate. Then it is clear to see that every irreducible component in 
$\partial \mathbb{P}(V^{\sharp})\cap D(L_{k-1})$ over $(z_{k},z_{k+1})$ is compact.

The proof of the proposition is complete.
\end{proof}

Now we are ready to establish the metric completion of $\mathbb{P}(G\backslash V^{\circ})$.

\begin{proposition}
In this case, the metric completion of $\mathbb{P}(G\backslash V^{\circ})$ is $\mathbb{P}(G\backslash V^{\natural})$, 
the Stein factor of the (multivalued) map from the space $\mathbb{P}(G\backslash(V^{\sharp}-\partial V^{\sharp}))$ 
to the complex ball $\mathbb{B}$. 
Moreover, the stratum after the birational operations induced by the projective structure 
is totally geodesic in $\mathbb{P}(G\backslash V^{\natural})$.
\end{proposition}

\begin{proof}
We still first look at the situation when $G$ is trivial. According to Proposition $\ref{prop:hyp-compact-fiber}$, 
the (multivalued) mapping $\mathbb{P}(V^{\sharp})\rightarrow \mathbb{B}^{\sharp}$ enjoys the property that 
the connected components of its fibers are compact and that the preimage of the Borel-Serre boundary is in the 
boundary of the domain. So we find that for the topological Stein factorization of 
$\mathbb{P}(V^{\sharp})\rightarrow \mathbb{B}^{\sharp}$,
\[
\mathbb{P}(V^{\sharp})\longrightarrow\mathbb{P}(V^{\sharp})_{\mathfrak{St}}\longrightarrow \mathbb{B}^{\sharp},
\]
we have a complex-analytic Stein factorization over $\mathbb{B}$ 
\[
\mathbb{P}(V^{\sharp}-\partial V^{\sharp})\longrightarrow\mathbb{P}(V^{\sharp}-\partial V^{\sharp})_{\mathfrak{St}}
\longrightarrow \mathbb{B}.
\]
We denote the Stein factor $\mathbb{P}(V^{\sharp}-\partial V^{\sharp})_{\mathfrak{St}}$ by $\mathbb{P}(V^{\natural})$.

This tells us that the extension across $E(L_{\bullet})^{\circ}$ (the stratum without boundary) 
by the projective developing map establishes a multivalued map 
from $\mathbb{P}(V^{\sharp}-\partial V^{\sharp})$ to $\mathbb{B}$. 
Observe that the $\Gamma$-invariant Hermitian form of hyperbolic type takes negative values 
on $A_{\mathbb{B}}$ so that $\mathbb{B}$ inherits a complex hyperbolic metric from $A_{\mathbb{B}}$. 
This extended map is certainly not injective near a point $z$ on $E(L_{\bullet})^{\circ}$ even up to the local monodromy 
group $\Gamma_{z}$, but we have the above complex-analytic Stein factorization 
for $\mathbb{P}(V^{\sharp}-\partial V^{\sharp})\rightarrow \mathbb{B}$ such that the differential of the derived map from 
the Stein factor $\mathbb{P}(V^{\natural})$ to $\mathbb{B}$ at $z$ is injective with the transition maps (between charts 
around $z$) lying in $\Gamma_{z}$, hence the resulting extended map is a local immersion up to $\Gamma$. 
This factorization is realized by a projection $E(L_{\bullet})^{\circ}\rightarrow \mathbb{P}((L_{r}/L_{r-1})^{\circ})$. 
This can be obtained by removing the boundary components $\partial_{L}\mathbb{P}(V^{\sharp})$ from $\mathbb{P}(V^{\sharp})$ 
if $\kappa_{L}=1$, 
contracting the transversal direction of $\mathbb{P}(L^{\sharp})$ for each 
$D(L)=\mathbb{P}(L^{\sharp})\times \mathbb{P}((V/L)^{\sharp})$ if $\kappa_{L}<1$, and contracting 
its own direction of $\mathbb{P}(L^{\sharp})$ for each $D(L)=\mathbb{P}(L^{\sharp})\times \mathbb{P}((V/L)^{\sharp})$ 
if $\kappa_{L}>1$. 

Then $\mathbb{P}(V^{\natural})$ is endowed with a (singular) complex hyperbolic metric pulled back from $\mathbb{B}$ 
since the resulting extended developing map is a local immersion up to $\Gamma$. 
It is clear that on $\mathbb{P}(V^{\circ})$ it is endowed with a smooth complex hyperbolic metric which coincides with 
the metric inherited from the given admissible Hermitian form of hyperbolic type on $V^{\circ}$. 
We further notice that this induced metric would not alter the original topology on $\mathbb{P}(V^{\natural})$, 
since the metric still varies over $\mathbb{P}(V^{\natural})$ continuously. 
We consider the following two cases. 1) If there is no $L$ such that $\kappa_{L}=1$, then $\mathbb{P}(V^{\natural})$ is compact. 
Since $\mathbb{P}(V^{\circ})$ is dense in $\mathbb{P}(V^{\natural})$, 
it implies that $\mathbb{P}(V^{\natural})$ is the metric completion of $\mathbb{P}(V^{\circ})$. 
2) If there happens to exist an $L$ whose $\kappa_{L}$ is $1$, 
then around the removed boundary component, it is easy to see from the projective structure $[F_{0},\log t, F_{1}]$ that 
each such boundary component is foliated by complex geodesics, locally isometric to a real hyperbolic plane, 
for which those real geodesics converge in an exponential manner, which shows that the metric is complete.  
This also implies that $\mathbb{P}(V^{\natural})$ is the metric completion of $\mathbb{P}(V^{\circ})$. 

The given admissible Hermitian form of hyperbolic type endows $\mathbb{P}(V^{\circ})$ with a K\"{a}hler metric of constant negative 
holomorphic sectional curvature, i.e., the complex hyperbolic metric. 
Since the Hermitian form is flat for the Dunkl connection and $\Gamma$-invariant, 
we still substitute the degenerate projective structure into the 
local smooth model $-\sqrt{-1}\partial\bar{\partial}\log (1-\sum_{i=1}^{n}|z_{i}|^{2})$, as in the elliptic case, 
we get the curvature models on strata of $\mathbb{P}(V^{\natural})$ as follows: 
$
-\sqrt{-1}\partial\bar{\partial}\log \big(1-(\sum_{i=1}^{d_{L}-1}|z_{i}|^{2})\big)
$
on $\mathbb{P}(L^{\circ})$ for which $\kappa_{L}<1$;
$
-\sqrt{-1}\partial\bar{\partial}\log \big(1-(\sum_{i=d_{L}+1}^{n}|z_{i}|^{2})\big)
$
on $\mathbb{P}((V/L)^{\circ})$ for which $\kappa_{L}>1$; and 
$
-\sqrt{-1}\partial\bar{\partial}\log \big(1-(\sum_{i=d_{L'}+1}^{d_{L}-1}|z_{i}|^{2})\big)
$
on $\mathbb{P}((L/L')^{\circ})$ for which $L'\subset L$ with $\kappa_{L'}>1>\kappa_{L}$, although the stratum $\mathbb{P}((L/L')^{\circ})$ 
is not irreducible in the sense that its transversal direction in $\mathbb{P}(V^{\natural})$ can be decomposed as a product. 
This shows that all the above strata are totally geodesic in $\mathbb{P}(V^{\natural})$.

When $G$ is not trivial, it is a finite group of projective unitary group of the Hermitian form of hyperbolic type 
generated by the Schwarz rotation group $G_{L}$. 
The group $G_{L}$ is a finite subgroup of the projective unitary group of hyperbolic type 
fixing $L$ (resp. $V/L$) and acting on $L^{\perp}$ (resp. $(V/L)^{\perp}$) as scalar multiplication by $|G_{L}|$th 
roots of unity when $\kappa_{L}<1$ (resp. $\kappa_{L}>1$) 
such that the Dunkl system is invariant under the action of $G_{L}$. 
We also know that the extended developing map is a constant on the $G$-orbits.
Therefore, the action of $G$ would not affect the metric completion of $\mathbb{P}(V^{\circ})$ or that the newly obtained stratum 
being totally geodesic in $\mathbb{P}(V^{\natural})$.

The proof of the proposition is now complete.
\end{proof}

\begin{proof}[Proof of Theorem $\ref{thm:hyperbolic}$]
It remains to show that the metric on $\mathbb{P}(G\backslash V^{\natural})$ is of a conical type and the finiteness of its volume. 

We still first deal with the situation when $G$ is trivial. 
For any point $z$ on $\mathbb{P}(V^{\natural})$, its unit tangent cone $S_{z}(\mathbb{P}(V^{\natural}))$ can be obtained in the 
same inductive way as in the elliptic case, which is a spherical cone-manifold whose strata in singular loci are of real odd dimension. 
So the complex hyperbolic metric cone $C_{g}(S_{z}(\mathbb{P}(V^{\natural})))$ over the unit tangent cone gives a local metric model 
for a neighborhood of $z$ such that the cones over those strata in singular loci in $S_{z}(\mathbb{P}(V^{\natural}))$ are totally geodesic 
in $C_{g}(S_{z}(\mathbb{P}(V^{\natural})))$. Therefore, the space $\mathbb{P}(V^{\natural})$ is a complex hyperbolic 
(i.e., $(\mathbb{CH}^{n},\mathrm{PU}(n,1))-$) cone-manifold.

If there is no $L$ with $\kappa_{L}=1$, the metric completion $\mathbb{P}(V^{\natural})$ is compact and hence the finiteness of the 
volume is automatic. When there exists an $L$ such that $\kappa_{L}=1$, it is easy to see 
from the projective structure $[F_{0},\log t, F_{1}]$ that each boundary component is foliated by complex geodesics with respect to the 
hyperbolic metric, and also that each complex geodesic is locally isometric to a real hyperbolic plane 
for which those real geodesics converge in an exponential manner, so the volume around each boundary component is finite. 
Therefore, $\mathbb{P}(V^{\natural})$ is a hyperbolic cone-manifold of finite volume.

Now we assume $G$ is not trivial, for $L\in \mathscr{L}_{\mathrm{irr}}(\mathscr{H})-\{0\}$ write $p_{L}:=|G_{L}|$ for which $G_{L}$ 
is the Schwarz rotation group fixing $L$ or $V/L$ depending on $\kappa_{L}<1$ or $\kappa_{L}>1$ as in the elliptic case. 
We note that the developing map near a point $z$ of $E(L)^{\circ}$ 
factors through $G_{L}\backslash V^{\sharp}_{z}$. 
Then we can also apply the inductive argument to 
$G\backslash L$ if $\kappa_{L}>1$, or $G\backslash (V/L)$ if $\kappa_{L}<1$, to get the conical structure along the stratum 
$\mathbb{P}(G\backslash (V/L)^{\circ})$ or $\mathbb{P}(G\backslash L^{\circ})$.

If some stratum in $\mathbb{P}(V^{\natural})$ has some rotation type group induced by $G$ while $G$ acts on its preimage in 
$V^{\natural}$ trivially, 
the stratum can be treated in the same way as in the elliptic case so that the conical structure remains along that stratum. 
Therefore, the conclusion follows as well by applying the same inductive argument to its normal fiber.

This completes the proof.
\end{proof}

\subsection{Some quantities}\label{subsec:link}

Next, we discuss some quantities associated to a stratum in the metric completion $\mathbb{P}(G\backslash V^{\natural})$. The discussion 
below holds for all three cases. We let $Q$ denote 
an irreducible stratum of codimension $q$, namely, it is of the form either $\mathbb{P}(L^{\circ})$ or $\mathbb{P}((V/L)^{\circ})$ 
for which we exclude the stratum of the form $\mathbb{P}((L/L')^{\circ})$ since its normal fiber is 
decomposed as $\mathbb{P}(L')\times \mathbb{P}(V/L)$. Then the normal slice to $Q$ at a point $z\in Q$ is a union of 
`complex rays' each of which is swept out by a real ray along a circle. The space of complex rays is the \emph{complex link} of 
the stratum, which is of a complex spherical cone-manifold structure of dimension $q-1$. Then a Seifert fiber space over the 
complex link with generic fiber a circle of length $\gamma(Q)$ is the real link of that stratum. We call the length $\gamma(Q)$ 
the \emph{scalar cone angle} at $Q$. It is clear that $\gamma(Q)=2\pi (1-\kappa_{L})$ if $Q=\mathbb{P}(L^{\circ})$, or 
$\gamma(Q)=2\pi (\kappa_{L}-1)$ if $Q=\mathbb{P}((V/L)^{\circ})$.

We define the \emph{complex link fraction} to be the ratio of the volume of the complex link to the volume of $\mathbb{P}^{q-1}$ 
(the complex link in the non-singular case), 
and likewise the \emph{real link fraction} to be the ratio of the volume of the real link to the volume of $S^{2q-1}$
(the real link in the non-singular case).

The complex resp. real link fraction of a given irreducible stratum in our situation 
is given as follows.

\begin{proposition}
Let $Q$ be an irreducible stratum in the metric completion $\mathbb{P}(G\backslash V^{\natural})$. 
Write $N_{Q}$ as the order of the corresponding 
Schwarz symmetry group of its transversal Dunkl system.

Then the complex link fraction is 
\[
\frac{(\gamma(Q)/2\pi)^{q-1}}{N_{Q}},
\]
and the real link fraction is 
\[
\frac{(\gamma(Q)/2\pi)^{q}}{N_{Q}}.
\]
\end{proposition}

\begin{proof}
We first consider the situation when $G$ is trivial. The complex link is of a cone-manifold structure on $\mathbb{P}^{q-1}$. 
If $\omega$ is a closed $2$-form on $\mathbb{P}^{q-1}$ such that it integrates to $1$ over $\mathbb{P}^{1}$, then 
$\omega^{q-1}$ gives the fundamental class for $\mathbb{P}^{q-1}$. Then for our cone-manifold structure we can readily 
obtain a volume form which is a constant multiple of the K\"{a}hler form of the non-singular standard $\mathbb{P}^{q-1}$. 
This calculus works since the differential forms associated to the cone metrics are suitably continuous.

We can determine the constant multiple by reducing the case to $q=2$. In the case $q=2$, the complex link is just $S^{2}$ 
with the cone points contributing the curvature $4\pi-2\gamma(Q)$. Since the total curvature of $S^{2}$ is $4\pi$, the 
area of the constant curvature metric is $2\gamma(Q)$. Hence its ratio to the volume of $S^{2}$ is $\gamma(Q)/2\pi$. 
So the complex link fraction is just $(\gamma(Q)/2\pi)^{q-1}$ for a general $q$.

The real link fraction is simply the product of the complex link fraction with $\gamma(Q)/2\pi$.

If $G$ is not trivial, the results follow immediately by dividing the symmetry.
\end{proof}

\section{Examples}\label{sec:examples}

In this section, we revisit the Deligne-Mostow theory (also studied by Thurston from a different angle) 
within the framework of the Dunkl system.

More specifically, the Deligne-Mostow theory \cite{Deligne-Mostow}\cite{Mostow} studies the monodromy problem of the 
Lauricella hypergeometric 
system, which leads to the discovery of ball quotient structures on $\mathbb{P}^{n}$ relative to a hyperplane configuration of 
type $A_{n+1}$. 
And later on, Thurston \cite{Thurston} reinterpreted the theory as the moduli of cone metrics on the 
sphere with a set of given curvatures on cone points, so that the metric completion of the moduli space acquires a cone-manifold 
structure (an orbifold appears to be a special kind of cone-manifold). 
This theory now turns out to be a special case in the set-up of the Dunkl system.

We start with an $(n+1)$-dimensional vector space defined by $V:=\mathbb{C}^{n+2}/\text{main diagonal}$, 
where $\text{main diagonal}$ is the line spanned by $(1,\dots,1)$. 
Let $\mathscr{H}$ be the collection of diagonal hyperplanes $H_{ij}\subset V$ (or in $\mathbb{C}^{n+2}$) defined by $z_{i}=z_{j}$ and
$\omega_{ij}:=(z_{i}-z_{j})^{-1}d(z_{i}-z_{j})$ the associated logarithmic form. 
So we can think of the arrangement complement $V^{\circ}:=V-\cup_{i<j}H_{ij}$ as the configuration space of 
$n+2$ ordered distinct points in $\mathbb{C}$, given up to translation.

Let be given a sequence of real numbers $\mu:=(\mu_{0},\dots,\mu_{n+1})$ such that each component $\mu_{i}\in (0,1)$, and write their sum as 
$|\mu|:=\sum_{i=0}^{n+1}\mu_{i}$. Denote the standard basis of $\mathbb{C}^{n+2}$ by $\varepsilon_{0},\dots,\varepsilon_{n+1}$. 
The inner product on $V$ comes from
the inner product on $\mathbb{C}^{n+2}$ defined by $(\varepsilon_{i},\varepsilon_{j})=\mu_{i}\delta_{j}^{i}$. In fact, we
could identify $V$ with the orthogonal complement of the main diagonal which is the hyperplane defined by
$\sum_{i}\mu_{i}z_{i}=0$. The line orthogonal to the hyperplane $H_{ij}$ is spanned by the vector
$\mu_{j}\varepsilon_{i}-\mu_{i}\varepsilon_{j}$. 
Now we define the endomorphism $\tilde{\rho}_{ij}:\mathbb{C}^{n+2}\rightarrow\mathbb{C}^{n+2};
z\mapsto (z_{i}-z_{j})(\mu_{j}\varepsilon_{i}-\mu_{i}\varepsilon_{j})$ which is selfadjoint with kernel $H_{ij}$. And 
a straightforward computation, 
$\tilde{\rho}_{ij}(\mu_{j}\varepsilon_{i}-\mu_{i}\varepsilon_{j})=(\mu_{i}+\mu_{j})(\mu_{j}\varepsilon_{i}-\mu_{i}\varepsilon_{j})$,
shows that $\kappa_{H_{ij}}=\mu_{i}+\mu_{j}$. In particular, $\tilde{\rho}_{ij}$ induces an endomorphism $\rho_{ij}$ in
$V$. 

One can verify that the connection defined by
\[ \nabla:=\nabla^{0}-\sum_{i<j}\omega_{ij}\otimes\rho_{ij}   \]
is flat so that the above $(V,\mathscr{H},\kappa)$ becomes a Dunkl system, and has the Euler field as a dilatation field 
with factor $1-|\mu|$.

Define the \emph{Lauricella differential} of weight $\mu$ as follows
\[
\eta_{z}:=(z_{0}-\xi)^{-\mu_{0}}\cdots (z_{n+1}-\xi)^{-\mu_{n+1}}d\xi,
\]
with $z=(z_{0},\dots,z_{n+1})\in(\mathbb{C}^{n+2})^{\circ}$. 
If a determination of $\eta_{z}$ is chosen and $\gamma$ is an arc in $\mathbb{C}$ connecting $z_{i}$ to $z_{j}$ which nevertheless 
does not meet any other point in $\{z_{0},\dots,z_{n+1}\}$, then the following integral 
\[
F(z,\gamma):=\int_{\gamma}\eta_{z}
\]
is translation invariant and thus defines a multivalued holomorphic function on $V^{\circ}$. 
All the functions of this form give rise to a new affine structure on $V^{\circ}$ 
for their differentials are flat with respect to the above connection. This also shows that together with constant functions, they 
form a solution space 
for a system of second-order differential equations which we call the \emph{Lauricella hypergeometric equations}. 
It is easy to verify that these functions are homogeneous of degree $1-|\mu|$.

Let $\delta:=(\delta_{1},\dots,\delta_{n+2})$ be a system of arcs for which $\delta_{k}$ denotes the piece connecting $z_{k-1}$ 
with $z_{k}$. We note that the arc $\delta_{n+2}$ connecting $z_{n+1}$ to $z_{n+2}:=\infty$ is required to follow the real axis in 
the positive direction when approaching $\infty$. 
The germs of these functions at $z\in V^{\circ}$ form a vector space of rank $n+1$ and we define functions as follows 
\begin{align*}
F_{k}(z,\delta)&:=\int_{\delta_{k}}(\xi-z_{0})^{-\mu_{0}}\cdots (\xi-z_{k-1})^{-\mu_{k-1}}(z_{k}-\xi)^{-\mu_{k}}\cdots 
(z_{n+1}-\xi)^{-\mu_{n+1}}d\xi \\
&=\bar{w}_{k}\int_{\delta_{k}}\eta_{z},
\end{align*}
for $k=1,\dots,n+2$, where $w_{k}:=\exp(\sqrt{-1}\pi(\mu_{0}+\cdots +\mu_{k-1}))$. 
We then have that these functions satisfy a linear relation $\sum_{k=1}^{n+2}\mathrm{Im}(w_{k})F_{k}=0$. 
This basis certainly defines a developing map for the new affine structure on $V^{\circ}$. 
The reason for us to choose this basis (from $F_{1}$ to $F_{n+1}$) is that in terms of this basis 
we are able to write out the required invariant Hermitian form in an explicit way.

Now let us find the Hermitian form on the target space $A$ which is identified with $\mathbb{C}^{n+1}$. 
We note that the required Hermitian form $H$ on the target space is characterized by 
\[
H(F(z,\delta),F(z,\delta))=N(z):=-\frac{\sqrt{-1}}{2}\int_{\mathbb{C}}\eta_{z}\wedge \bar{\eta}_{z}.
\]
According to this, with some effort we can write out a form on $\mathbb{C}^{n+2}$, denoted by 
$\tilde{H}$, as follows
\[
\tilde{H}(F,F')=\sum_{1\leq j<k \leq n+2}\mathrm{Im}(w_{j}\bar{w}_{k})F_{k}\bar{F}'_{j}.
\]
We denote by $\tilde{A}$ the hyperplane in $\mathbb{C}^{n+2}$ defined by $\sum_{k=1}^{n+2}\mathrm{Im}(w_{k})F_{k}=0$. 
When $|\mu|\notin \mathbb{Z}$, we can identify $\mathbb{C}^{n+1}$ with the projection $\tilde{A}\subset\mathbb{C}^{n+2}
\rightarrow \mathbb{C}^{n+1}$ by forgetting the last coordinate (this can be done because $\mathrm{Im}(w_{n+2})=\sin(|\mu|\pi)\neq 0$). 
So the restriction of $\tilde{H}$ on $\tilde{A}$ gives us the Hermitian form $H$ on $\mathbb{C}^{n+1}$. 
When $|\mu|\in \mathbb{Z}$, the projection $\tilde{A}\subset\mathbb{C}^{n+2}\rightarrow \mathbb{C}^{n+1}$ defines a 
hyperplane $A$ in $\mathbb{C}^{n+1}$ (given by $\sum_{k=1}^{n+1}\mathrm{Im}(w_{k})F_{k}=0$) since $\mathrm{Im}(w_{n+2})=0$. 
So the restriction of $\tilde{H}$ on $\tilde{A}$ induces a Hermitian form on $A$, 
which can also be viewed as being induced from a degenerate Hermitian form $H$ on $\mathbb{C}^{n+1}$. 
In either case the form $H$ is $\Gamma$-invariant.

We can also compute that $\kappa_{0}=|\mu|$; and if an irreducible member $L\in\mathscr{L}_{\mathrm{irr}}(\mathscr{H})$ is given by
a subset $I\subset\{0,\dots,n+1\}$ with at least two elements, i.e., $L=L(I)$ is the locus where all $z_{j},j\in I$ coincide, then
$\kappa_{L}=\sum_{j\in I}\mu_{j}$. We denote by $\mathfrak{S}_{\mu}$ the permutation group preserving the weight system $\mu$, 
which plays a role as the Schwarz symmetry group $G$ in the Dunkl system.

Now this example well falls into the setting of the Dunkl system that we discuss in the preceding sections. Then we have 

\begin{theorem}
Let $\mu$ be given as above. Then there is a flat Hermitian form $h$ on the tangent bundle of $V^{\circ}$ as given by $H$ 
on the target space realized by the developing map (it can be done since $H$ is $\Gamma$-invariant).
\begin{enumerate}[label=(\roman*)]
	\item \emph{(elliptic case)} If $|\mu|<1$, then the Hermitian form is positive definite which endows 
	$\mathfrak{S}_{\mu}\backslash\mathbb{P}(V^{\circ})$ 
	with a K\"{a}hler metric locally isometric to a Fubini-Study metric. Its metric completion acquires the structure of 
	an elliptic cone-manifold.
	
	\item \emph{(parabolic case)} If $|\mu|=1$, then the Hermitian form is positive semidefinite which endows 
	$\mathfrak{S}_{\mu}\backslash\mathbb{P}(V^{\circ})$ 
	with a flat K\"{a}hler metric. Its metric completion acquires the structure of 
	a parabolic cone-manifold.
	
	\item \emph{(hyperbolic case)} If $1<|\mu|<2$, then the Hermitian form is admissible and of hyperbolic type such that $h(z,z)<0$. Hence 
	this Hermitian form endows $\mathfrak{S}_{\mu}\backslash\mathbb{P}(V^{\circ})$ with a K\"{a}hler metric locally isometric to 
	a complex hyperbolic metric. 
	Its metric completion acquires the structure of a complex hyperbolic cone-manifold of finite volume.
\end{enumerate}
\end{theorem}

\bibliographystyle{alpha}
\bibliography{bibfile}
\end{document}